\documentclass[11pt,reqno]{amsart}

\usepackage{amsmath}
\usepackage{amssymb}
\usepackage{amsthm}
\usepackage{latexsym}
\usepackage{mathtools}
\mathtoolsset{showonlyrefs,showmanualtags}
\usepackage{thmtools}
\usepackage{amsfonts}
\usepackage{mathrsfs}
\usepackage{textcomp}
\usepackage{graphicx}
\usepackage{setspace}
\usepackage{nicefrac}
\usepackage{indentfirst}
\usepackage{enumerate}
\usepackage{wasysym}
\usepackage[normalem]{ulem}
\usepackage{upgreek}

\usepackage{paralist}
\usepackage{xcolor}
\usepackage{etoolbox}
\usepackage{mathdots}
\usepackage{wrapfig}
\usepackage{floatflt}
\usepackage{tensor} 
\usepackage{parskip}
\usepackage{lscape}
\usepackage{stmaryrd}
\usepackage[enableskew]{youngtab}
\usepackage{ytableau}
\usepackage{pifont}

\usepackage{tikz}
\usetikzlibrary{arrows}
\usetikzlibrary{matrix}
\usetikzlibrary{positioning}
\usetikzlibrary{snakes}
\usetikzlibrary{calc}

\usepackage[all,cmtip]{xy}
\usepackage[labelfont=bf, font={small}]{caption}[2005/07/16]

\usepackage{xcolor}
\definecolor{red}{rgb}{1,0,0}

\newcommand{\vvirg}{ , \dots , }

\newcommand{\ttimes}{ \times \cdots \times }

\newcommand{\textsum}{{\textstyle \sum}}
\newcommand{\textprod}{{\textstyle \prod}}

\newcommand{\textfrac}[2]{{\textstyle \frac{#1}{#2}}}

\newcommand{\bfN}{\mathbf{N}}

\newcommand{\calB}{\mathcal{B}}

\newcommand{\calD}{\mathcal{D}}

\newcommand{\calI}{\mathcal{I}}

\newcommand{\calL}{\mathcal{L}}
\newcommand{\calM}{\mathcal{M}}

\newcommand{\calO}{\mathcal{O}}

\newcommand{\calS}{\mathcal{S}}

\newcommand{\calU}{\mathcal{U}}
\newcommand{\calV}{\mathcal{V}}

\newcommand{\calY}{\mathcal{Y}}
\newcommand{\calZ}{\mathcal{Z}}

\newcommand{\bbC}{\mathbb{C}}

\newcommand{\bbN}{\mathbb{N}}

\newcommand{\bbP}{\mathbb{P}}

\newcommand{\bbR}{\mathbb{R}}

\newcommand{\frakB}{\mathfrak{B}}

\newcommand{\bfxi}{\boldsymbol{\xi}}

\renewcommand{\phi}{\varphi}

\newcommand{\dashto}{\dashrightarrow}

\renewcommand{\tilde}[1]{\widetilde{#1}}

\renewcommand{\bar}[1]{\overline{#1}}

\newcommand{\rank}{\mathrm{rank}}

\DeclareMathOperator{\codim}{codim}

\newcommand{\Mat}{\mathrm{Mat}}

\usepackage[T1]{fontenc}

\usepackage{ dsfont }

\usepackage[top=3.1cm,bottom=3.2cm,left=3cm,right=3cm,marginparwidth=0cm]{geometry}
\usepackage[colorlinks=true]{hyperref}
 \hypersetup{
     colorlinks,
     citecolor={blue!50!black},
     linkcolor={red!50!black},
     urlcolor={green!80!black}
 }

\usepackage{enumitem} 
\usepackage{tikz-cd}
\usepackage{subcaption}

\newtheorem{theorem}{Theorem}[section]
\newtheorem{proposition}[theorem]{Proposition}
\newtheorem{corollary}[theorem]{Corollary}
\newtheorem{lemma}[theorem]{Lemma}
\newtheorem{conj}[theorem]{Conjecture}

\theoremstyle{definition}
\newtheorem{definition}[theorem]{Definition}
\newtheorem{problem}[theorem]{Problem}
\newtheorem{remark}[theorem]{Remark}
\newtheorem{example}[theorem]{Example}

\newcommand{\N}{\mathbb{N}}

\renewcommand{\int}{\operatorname{int}}

\newcommand{\lprod}{\scalebox{1.2}{\guilsinglleft}}
\newcommand{\rprod}{\scalebox{1.2}{\guilsinglright}}

\DeclareMathOperator{\conv}{conv}
\DeclareMathOperator{\crit}{crit}

\title{The Geometry of Discotopes}
\author{Fulvio Gesmundo and Chiara Meroni}

\address[F. Gesmundo]{Max Planck Institute for Mathematics in the Sciences, Inselstrasse 22, 04103 Leipzig, Germany -- (current) Saarland University, Saarbr\"ucken, Germany}
\email{gesmundo@cs.uni-saarland.de}

\address[C. Meroni]{Max Planck Institute for Mathematics in the Sciences, Inselstrasse 22, 04103 Leipzig, Germany}
\email{chiara.meroni@mis.mpg.de}

\subjclass[2010]{(primary) 14P10, 52A99; (secondary) 52A21, 14M12}
\keywords{discotope, zonoid, convex body, algebraic boundary, determinantal variety}

\date{}

\begin{document}

\begin{abstract}
We study a class of semialgebraic convex bodies called discotopes. These are instances of zonoids, objects of interest in real algebraic geometry and random geometry. We focus on the face structure and on the boundary hypersurface of discotopes, highlighting interesting birational properties which may be investigated using tools from algebraic geometry. When a discotope is the Minkowski sum of two-dimensional discs, the Zariski closure of its set of extreme points is an irreducible hypersurface. In this case, we provide an upper bound for the degree of the hypersurface, drawing connections to the theory of classical determinantal varieties. 
\end{abstract}

\maketitle

\section{Introduction}

Discotopes are finite Minkowski sums of generalized discs in the real Euclidean space. They were introduced in \cite{AdiSan:WhitneyNumbersArrangements} for the combinatorial study of matroids associated to subspace arrangements. They appeared in the context of convex geometry in \cite{MatMer:FiberConvexBodies}, where they provide an example of a semialgebraic convex body having a  non-semialgebraic fiber body.

In this work, we investigate geometric features of discotopes. Our study is motivated by the \emph{zonoid problem}, introduced in \cite{Bol:ZonoidProblem}, but already appearing in \cite{Blas:Differentialgeometrie}: this problem consists in determining whether a given convex body is a zonoid. These special convex sets play an important role in convex geometry, measure theory, functional analysis and random geometry  \cite{Bol:ClassConvexBodies,SchWeil:ZonoidsRelatedTopics,Vit:ExpectedAbsRandDetZonoids,GoodWei:ZonoidsGeneralizations}. More recently, connections to enumerative geometry and real intersection theory were drawn \cite{BuerLer:ProbabilisticSchubertCalculus}, which led to the introduction of the zonoid algebra in \cite{BreBueLerMat:ZonoidAlgebra} as a probabilistic version of cohomology.

Zonoids are limits, in the Hausdorff metric, of zonotopes. The latter are finite Minkowski sums of line segments. Zonotopes are the only polytopes that are zonoids. They are relatively well-understood: for instance, it is known that a polytope is a zonotope if and only if all its $2$-dimensional faces are (translates of) centrally symmetric polygons \cite{Bol:ClassConvexBodies,Schn:UberIntegralTheorieKonvexenKorper}. On the other hand, a simple characterization of zonoids seems hopeless, and in full generality even the decidability of the zonoid problem is not understood. We highlight two major difficulties. In \cite{Weil:BlaschkesLocalZonoid}, it was shown that being a zonoid is not a local property in the sense that for every convex body $K$, and every point $p$ of its boundary $\partial K$, there is a zonoid whose support function coincides with the support function of $K$ in a neighborhood of $p$. Moreover, in \cite{Wil:ZonoideVerwandteKlassen}, it was shown that being a zonoid is not a property characterized by projections; indeed there exist convex bodies which are not zonoids but such that all their projections are zonoids.

Restricting to the subclass of semialgebraic convex bodies would potentially make the problem easier. For instance, in the algebraic setting, the rigidity of the Zariski topology implies that global properties can be checked locally. The tameness of a particular class of semialgebraic zonoids was proved in \cite{LerMat:TamenessZonoids}. However, the zonoid problem is open even in the simplest non-trivial case of semialgebraic convex bodies in $\bbR^3$, see e.g. \cite[Problem 12]{Sturm:TwentyFacets}.

Discotopes form a special subclass of semialgebraic zonoids. They are a first possible generalization of zonotopes, still amenable to be studied with tools from algebra, geometry and combinatorics. A discotope is a finite Minkowski sum of (generalized) higher dimensional discs: from this point of view, zonotopes correspond to the special case of $1$-dimensional discs.

In this work, we investigate a number of properties of discotopes. Section \ref{sec: general setting} provides the formal definition of discotopes and states some of their basic features. In Section \ref{sec: faces}, we describe the facial structure of these convex bodies, and we introduce a particular subvariety $\calS$ of their (algebraic) boundary, whose properties are further studied in the rest of the paper. Section \ref{sec: joins} provides a full characterization of $\calS$ in a special range. Section \ref{sec:exposed_points} describes its role in the geometry of the exposed points of the discotope. 
In Section \ref{sec:2discs}, we study discotopes that are Minkowski sums of $2$-dimensional discs; we prove that in this case $\calS$ is an irreducible hypersurface and provide an upper bound for its degree. As a by-product of this result, we prove that certain non-generic linear sections of the classical determinantal variety are irreducible and of the expected dimension. 
Section \ref{sec: dice} is devoted to the study of a specific example, the dice (see Figure \ref{fig:dice}), which presents peculiar birational properties. Finally, in Section \ref{sec: conclusions}, we propose some open problems and conjectures: in particular, Conjecture \ref{conj:irr_S} predicts that the variety $\calS$ is irreducible under minimal assumptions.

\begin{figure}[ht!]
\centering
\includegraphics[width=0.3\textwidth]{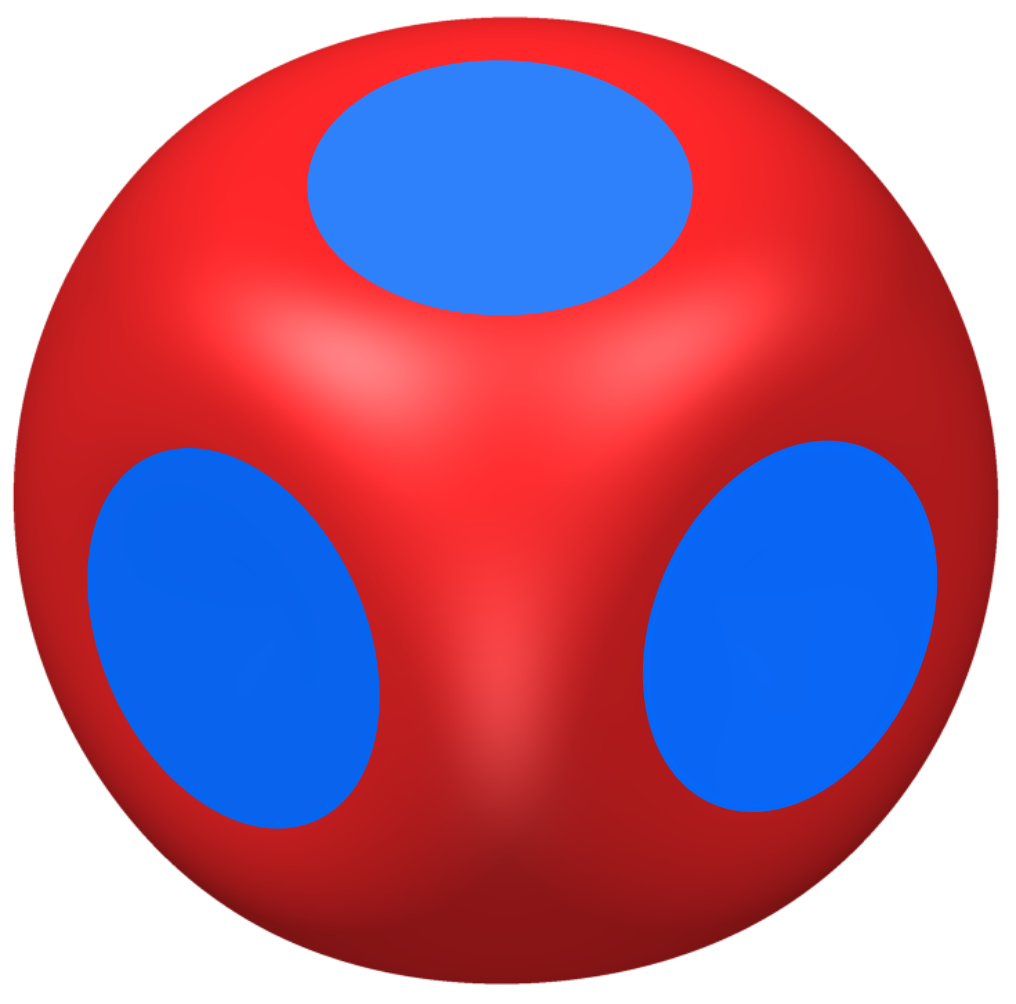}
\caption{The dice: the Minkowski sum of three discs in $\bbR^3$.} \label{fig:dice}
\end{figure}

\subsection*{Acknowledgements}
We would like to thank Daniele Agostini for our enlightening discussions on the applications of Bertini's Theorem. We are grateful to Rainer Sinn and Bernd Sturmfels for their frequent and useful comments during the development of this project. Finally, we thank Marie-Charlotte Brandenburg, Paul Breiding, Lukas K\"uhne, Simone Naldi, Felix Rydell and Simon Telen for helpful discussions and feedback.

\section{Setting}\label{sec: general setting}

We recall some basic notions from convex geometry and we refer to \cite{Sch:ConvexBodies} for the theory.  Let $S^{d-1}$ denote the $(d-1)$-dimensional sphere in $\bbR^d$. Given a convex body $K \subseteq \bbR^d$, its support function is 
\begin{equation}
\begin{aligned}
 h_K : S^{d-1} &\to \bbR \\
 u &\mapsto \max\{ \lprod u , p\rprod : p \in K\},
\end{aligned}
\end{equation}
where $\lprod \cdot \, , \cdot \rprod$ denotes the standard inner product of $\bbR^d$. The support function of $K$ uniquely determines $K$; moreover, support functions are additive with respect to the Minkowski sum, in the sense that $h_{K_1 + K_2} = h_{K_1} + h_{K_2}$. 
The set $K^u = \{ p \in K : \lprod u , p \rprod = h_K(u)\}$ is the \emph{face} of $K$ exposed by $u$; clearly $K^u$ is a subset of the topological boundary of $K$; if $p \in K^u$, we say that $p$ is exposed by $u$. In particular, if $\{p\} = K^u$, then $p$ is called exposed point. Let $\partial K \subseteq \bbR^d$ denote the topological boundary of $K$. If $K$ is a semialgebraic convex body, then $\partial K$ is a semialgebraic set of codimension one in $K$.

Many geometric features of semialgebraic convex bodies can be studied with tools from classical algebraic geometry. The \emph{algebraic boundary} represents this transition from real convex geometry to complex algebraic geometry.
It is the Zariski closure of $\partial K$ in $\bbC^d$, that will be denoted by  $\partial_a K$.
A full dimensional convex body $K$ is semialgebraic if and only if $\partial_a K$ is a hypersurface \cite[Proposition 2.9]{Sinn:AlgebraicBoundaries}. 

In this section we introduce discotopes and highlight some of their basic properties.
Given $n\in \bbN$, let $D$ be the standard unit ball in $\bbR^n$, that is
\begin{equation}
D= \{ (x_1 \vvirg x_n) : x_1^2+\ldots+x_n^2\leq 1\}.
\end{equation}
Fix $n,d \in \bbN$, $n\leq d$, let $\calB: = \{b_1,\ldots ,b_n\}$ be a set of linearly independent vectors of $\bbR^d$ and let $A_{\calB} : \bbR^n \to \bbR^d$ be the linear map mapping the $i$-th standard basis element $e_i$ of $\bbR^n$ to $b_i$. The \emph{generalized disc} $D_{\mathcal{B}}$ is the image of $D$ via $A_{\calB}$. Throughout, generalized discs are simply called discs.

The topological boundary $\partial D_{\calB}$ is a real algebraic hypersurface in the linear span $\langle \calB \rangle$: its ideal is defined by $d-n$ linear forms determining $\langle \calB \rangle$ and a single inhomogeneous quadric $q_\calB - 1$, where $q_\calB$ is the quadratic form associated to the matrix $(A_{\calB}) (A_{\calB})^T$. In particular, generalized discs are semialgebraic sets. 

\begin{remark}
For every $d$ and every choice of $\mathcal{B}$, the generalized disc $D_{\mathcal{B}}$ is a zonoid. This is immediate from the fact that linear images of zonoids are zonoids. In particular for $d=1$, generalized discs are all the compact segments centered at the origin; for higher $d$, generalized discs are ellipsoids centered at the origin.
\end{remark}

\begin{definition}
Given the generalized discs $D_{\calB_1}, \ldots, D_{\calB_N}$ in $\bbR^d$,
the \emph{discotope} $\calD_{\frakB}$ associated to $\frakB = \{ \calB_{j} \,|\,  j= 1 \vvirg N\}$ is their Minkowski sum
\begin{equation}
 \calD_{\frakB} = D_{\calB_1} + \cdots + D_{\calB_N}.
\end{equation}
Write $D_i := D_{\calB_i}$ if no confusion arises.
Let $N_m$ be the number of discs of dimension $m$ among $D_1, \ldots, D_N$. The \emph{type} of the discotope $\calD_{\mathfrak{B}}$ is the integer vector $\bfN = (N_1,\ldots, N_{d}) \in \N^{d}$. Note that $N = \sum N_m$.
\end{definition}

Usually, we will be interested in the case where the sets $\calB_{j}$ are chosen generically. More precisely, we say that a property holds for the \emph{generic discotope} of type $\bfN$ if the sets $\frakB$ for which it does not hold form a proper Zariski closed subset of the set of all possible bases.

In the case $\bfN = (N, 0, \ldots, 0)$, all the generalized discs are segments centered at the origin: the associated discotope is a zonotope centered at the origin \cite[Section 7.3]{Zieg:LecturesPolytopes}. The case $\bfN = (0,N,0)$ of discotopes in $\bbR^3$ was studied in \cite{MatMer:FiberConvexBodies} in the context of fiber convex bodies. In \cite{AdiSan:WhitneyNumbersArrangements}, discs of higher dimensions were considered, suitably rescaled so that their volume is normalized.

Discotopes can be realized as the image of the addition map restricted to the product of the discs. More precisely, define $\Sigma$ to be the (complex) \emph{addition map}
\begin{align*}
 \Sigma: (\bbC^d)^{N} &\to \bbC^d \\
 (\xi_{j})_{j= 1 \vvirg N} &\mapsto \textsum \xi_{j}.
\end{align*}
Then, the discotope $\calD$ associated to $\frakB$ is the image of $\prod_{j} D_{j} \subseteq (\bbR^d)^{N}$ under $\Sigma$. In particular, the Projection Theorem for semialgebraic sets (see, e.g., \cite[Section 2.2]{BCR:RealAG}) guarantees that $\calD$ is semialgebraic. 

Since Minkowski sums of zonoids are zonoids, every discotope is a zonoid. In particular, discotopes form a class of semialgebraic zonoids and one may wonder whether all semialgebraic zonoids centered at the origin arise in this way. This is not the case. An example of a semialgebraic zonoid which is not a discotope is the unit ball of the $L^4$-norm $\{ x_1^4 + x_2^4 \leq 1\}$ in $\bbR^2$: indeed, there is no discotope in $\bbR^2$ whose boundary is an irreducible curve of degree $4$, see Remark \ref{rmk:L4ball}.

We point out that a discotope is full dimensional if and only if $\sum_{j} \langle \calB_j \rangle = \bbR^d$. In particular, a necessary condition for this to happen is that $\sum_1^d m N_m \geq d$. For a generic discotope this condition is also sufficient. We always assume that $\calD$ is full dimensional: this is not restrictive as one can always restrict the analysis to the linear span $H  = \sum_{j} \langle \calB_j \rangle = \langle \calD \rangle$.

\section{The faces of the discotope}\label{sec: faces}

In this section we give a characterization of the exposed faces of discotopes. Moreover, we introduce a complex algebraic variety associated to a discotope, called its purely nonlinear part, which will be the main object of study in the rest of the paper.

Consider the discotope $\calD = D_1 + \ldots + D_N$.
For every disc $D_{j}$, let $C_{j} = S^{d-1} \cap \langle D_{j} \rangle ^\perp$, that is the unit sphere of dimension $d - \dim D_{j}-1$ consisting of directions orthogonal to $D_{j}$. Let $\calU = S^{d-1}\setminus \left( \bigcup_j C_{j} \right)$, which is Zariski open in $S^{d-1}$. If $u \in \calU$, then for every disc $D_j$ the face $D_j^u$ exposed by $u$ is a single point. 

As a consequence, if $u \in \calU$, then the face of the discotope $\calD^u$ consists of a single point. To see this, let $p =  \sum \xi_j \in \calD^u$ be a point of the face exposed by $u$. Then
\begin{equation}
h_{\calD}(u) = \lprod u, p \rprod = \sum_{j} \lprod u, \xi_{j} \rprod > \sum_{j} \lprod u, \tilde{\xi}_{j} \rprod
\end{equation}
for any other $\tilde{\xi}_{j} \in D_{j}$. Therefore $\calD^u = \{ p \}$ and hence $p$ is an extreme exposed point of $\calD$.

On the other hand if $u \notin \calU$, let $J \subseteq \{ 1 \vvirg N\}$ be the maximal subset of indices such that $u \in \bigcap_{j \in J} C_{j}$; then $u^{\perp}$ contains $ \sum_{j \in J}\langle \calB_j \rangle$. In this case the face of $\calD$ exposed by $u$ is (a properly translated copy of) a smaller discotope $\calD'$, given by the Minkowski sum of the discs $D_{j}$ for $j\in J$. From this description, one verifies a property that holds more generally for zonoids: every proper face of a zonoid $Z$ is a translate of a zonoid of lower dimension, which is a summand of $Z$.

In particular, the exposed faces of $\calD$ of dimension $k$ are given by 
\begin{equation}
\sum_{j\in J} D_{j} + \sum_{i \notin J} \{p_i\}
\end{equation}
where $p_i\in \partial D_{i}$ are certain suitable points and $J$ is such that $\dim \big(\sum_{j \in J} \langle \calB_j \rangle\big)  = k$.

\begin{remark}\label{rmk:zonotope_discotope}
A discotope $\calD$ of type $\bfN = (N_1 \vvirg N_{d})$ is the Minkowski sum of a zonotope $\calZ$ given by $N_1$ segments and a discotope $\calD'$ of type $(0 , N_2 \vvirg N_{d})$:
\begin{equation}
\calD = \calZ + \calD'.
\end{equation}
Since the convex hull of a Minkowski sum equals the Minkowski sum of convex hulls, the discotope $\calD$ is the convex hull of copies of $\calD'$ placed at the vertices of $\calZ$. As a consequence, many of the geometric properties of $\calD$ only depend on analogous properties of $\calD'$. For instance, the algebraic study of extreme points of $\calD$ can be reduced to the one of extreme points of $\calD'$. This can be visualized in the following example.
\end{remark}

\begin{example}\label{ex:discotopes_segment}
Let $D_1, D_2, D_3$ be the discs in $\bbR^3$ defined by
\begin{align}
D_1 &= \{(x_1,x_2,x_3) : x_3 = 0, -1 \leq x_1=x_2 \leq 1\}, \\
D_2 &= \{(x_1,x_2,x_3) : x_1 = 0, x_2^2+x_3^2\leq 1\}, \\
D_3 &= \{(x_1,x_2,x_3) : x_2 = 0, x_1^2+x_3^2\leq 1\}.
\end{align}
Consider the associated discotope $\calD = D_1 +D_2 + D_3$, shown in Figure \ref{fig:space_discotopes}, left. Faces of dimension $0, 1$ and $2$ are represented in red, green and blue, respectively. The red points are exposed and arise as $\xi_1+\xi_2+\xi_3$, for certain $\xi_i \in \partial D_i$. Every green segment arises as $D_1 + \xi_2 + \xi_3$, for certain $\xi_i \in \partial D_i$. The four blue discs (only two of which are visible) come in pairs: two are obtained as $\xi_1+D_2+\xi_3$ and the other two as $\xi_1+\xi_2+D_3$, for certain $\xi_i \in \partial D_i$.
\begin{figure}[ht!]
\centering
\begin{subfigure}{0.49\textwidth}
\centering
\includegraphics[width=0.75\textwidth]{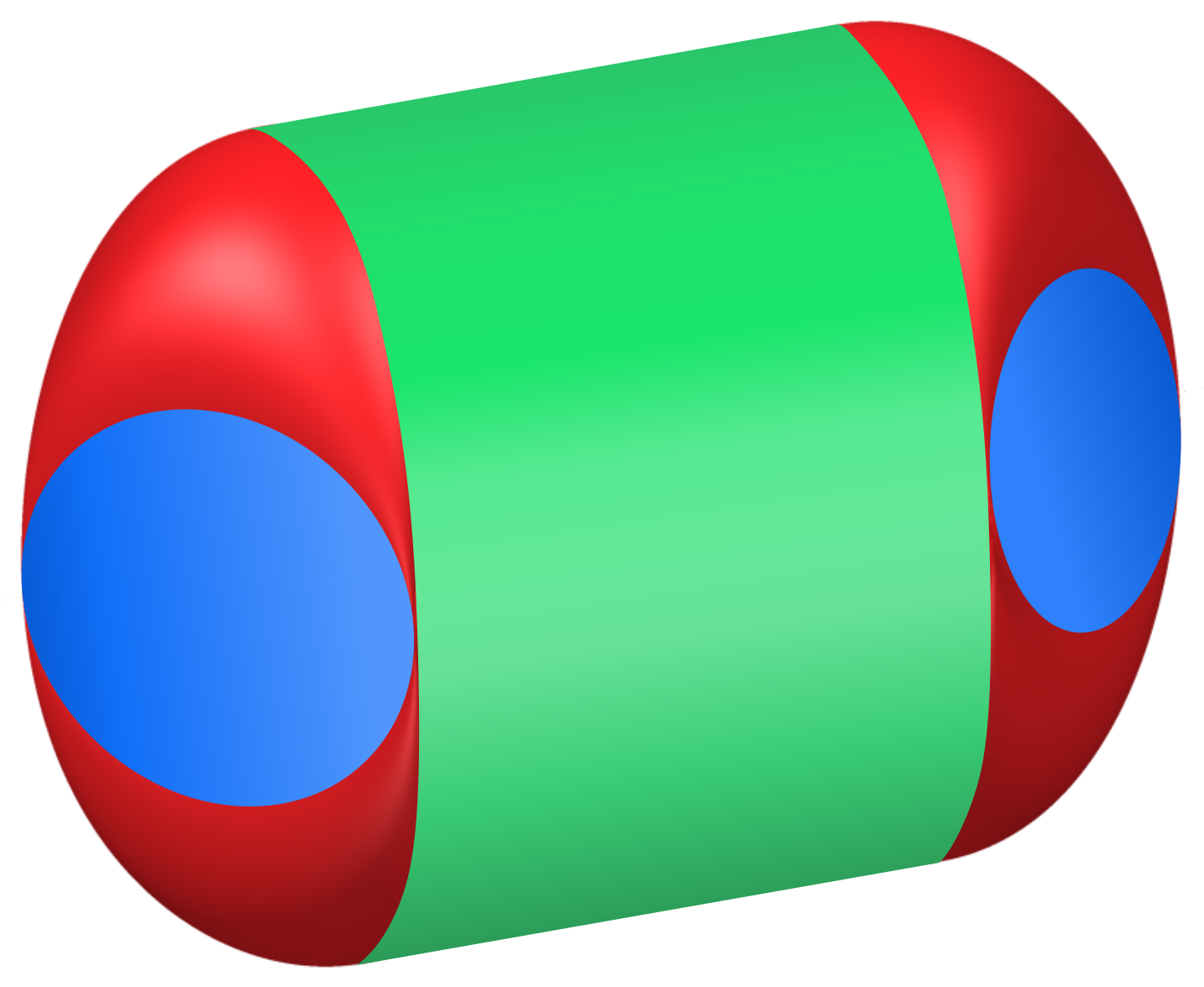}
\label{fig:d_1,2,0}
\end{subfigure}
\begin{subfigure}{0.49\textwidth}
\centering
\includegraphics[width=0.4\textwidth]{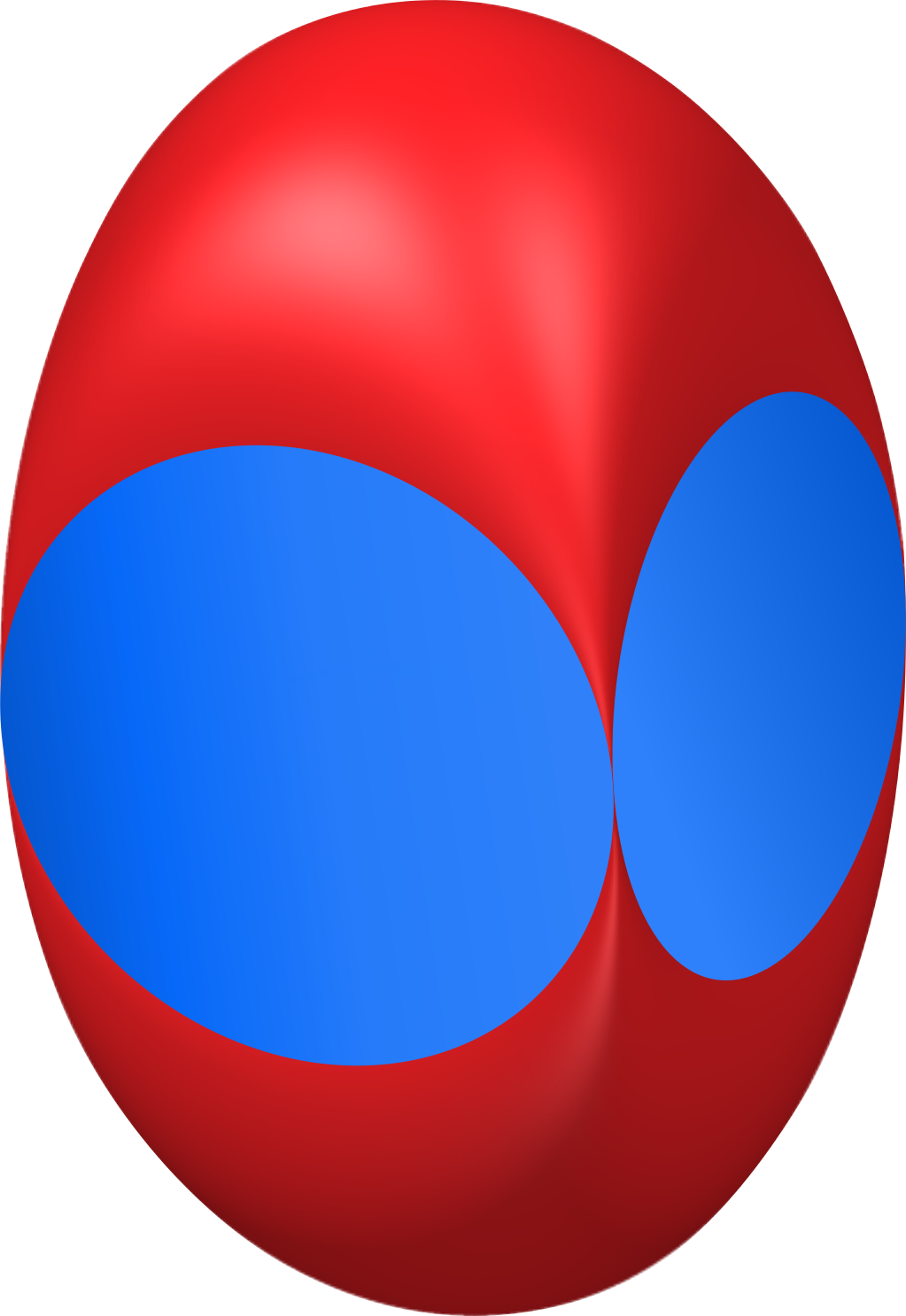}
\end{subfigure}
\caption{Left: a discotope of type $\bfN = (1,2,0)$. Right: a discotope of type $\bfN~=~(0,2,0)$. Faces of dimension $0$ are in red, faces of dimension $1$ are in green and faces of dimension $2$ are in blue.}
\label{fig:space_discotopes}
\end{figure}
\end{example}
As observed in Remark \ref{rmk:zonotope_discotope}, $\calD = \calZ + \calD'$ where $\calZ = D_1$ is a zonotope and $\calD' = D_2+D_3$ is a discotope with $N_1 = 0$, shown in Figure \ref{fig:space_discotopes}, right. The algebraic boundary $\partial_a \calD'$ consists of five irreducible components: four planes and the quartic surface
\begin{equation}
\calS = \{ x_1^4-2 x_1^2 x_2^2+x_2^4+2 x_1^2 x_3^2+2 x_2^2 x_3^2+x_3^4-4 x_3^2 = 0 \}.
\end{equation}
The latter is the Zariski closure of the set of extreme points of $\calD'$. Instead, the Zariski closure of the set of exposed points of $\calD$ is the union of two copies of $\calS$, translated by the extrema of $D_1$, i.e., the vectors $\pm (1,1,0)$.

Next, we associate to $\calD$ an complex algebraic variety that contains all extreme points of the discotope. 
\begin{definition}\label{def:S}
Let $\calD^\partial = \Sigma \left( \textprod_{j} \partial D_{j} \right) \subseteq \calD$. The \emph{purely nonlinear part} of the discotope $\calD$ is 
\begin{equation}
\calS = \overline{\calD^{\partial} \cap \partial \calD},
\end{equation}
the Zariski closure of $\calD^{\partial} \cap \partial \calD$ in $\bbC^d$.
\end{definition}
 By definition if $p$ is an extreme point of $\calD$, then $p\in \calS$. Therefore, by the Krein--Milman Theorem \cite[Section II.3]{Barvinok:ConvexityBook}, 
$\calD = \conv ( \calS \cap \partial \calD )$ is the convex hull of $\calS \cap \partial \calD$.

In particular, $\calS$ carries all the information regarding the extreme points of $\calD$. In general, the variety $\calS$ may have several irreducible components, possibly of different dimension. In fact, it is a priori not clear whether $\calS$ coincides with the Zariski closure of the set of exposed point of $\calD$. We will prove some results in this direction in Section \ref{sec:exposed_points}.
In particular, Corollary \ref{cor:dim_S} guarantees that when the discs are chosen generically, $\calS$ has dimension $d-1$ (possibly with lower dimensional components) in $\bbC^d$ if and only if the following non-degeneracy condition holds:
\begin{equation}\label{eq:condition_S}
\sum_{m=1}^{d} (m-1) N_m\geq d-1.
\end{equation}
Notice that this condition implies the non-degeneracy condition $\sum_1^{d} m N_m \geq d$ for the discotope, and it is immediately satisfied if $N_d\geq 1$. 

\begin{remark}\label{rmk:irr_top_discs}
In the special case of discotopes of type $\bfN = (0 \vvirg 0, N)$, all the boundary points of $\calD$ are exposed and therefore $\calS$ coincides with the Zariski closure of the set of exposed points. Further, $\partial \calD$ is smooth (see, e.g., \cite{BJ:SmoothnessMinkowskiSum}), which guarantees that $\calS$ is irreducible: indeed, if it was reducible, any two irreducible components would intersect on $\partial \calD$, in contradiction with its smoothness. Figure \ref{fig:planar_discotopes} shows a discotope $\calD$ obtained as the sum of three ellipses in $\bbR^2$.
The topological boundary $\partial \calD$ is smooth and coincides with one of the connected components of the real locus of $\partial_a \calD$. These properties are further discussed in Section \ref{sec:2discs}.
\begin{figure}[ht!]
\centering
\includegraphics[width=0.3\textwidth]{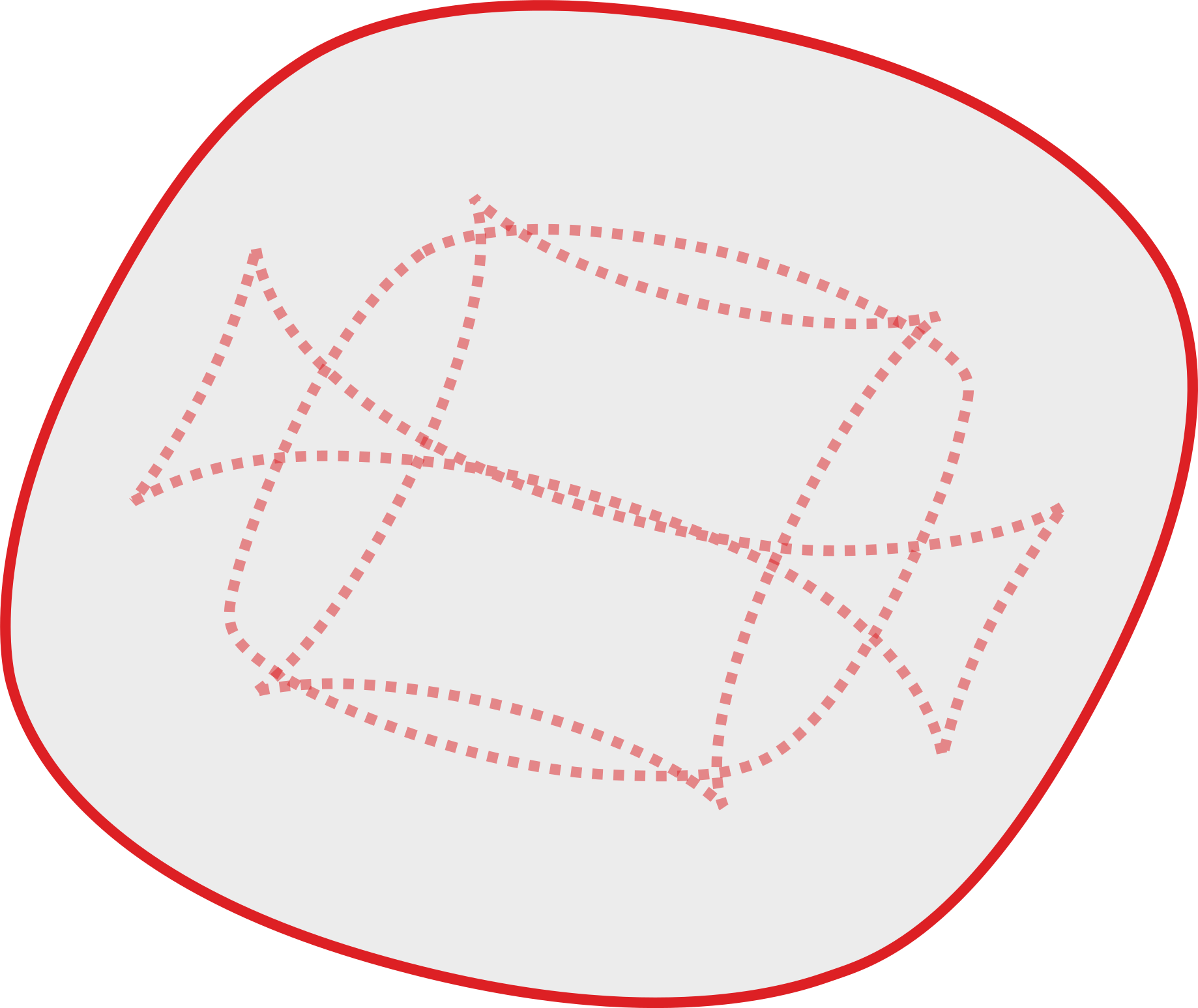}
\caption{A discotope of type $\bfN = (0,3)$. Its algebraic boundary is an irreducible curve of degree $24$. The dashed curves represent the real points of the algebraic boundary that are not part of the topological boundary.}
\label{fig:planar_discotopes}
\end{figure}
\end{remark}

\section{Joins of quadrics}\label{sec: joins}

If \eqref{eq:condition_S} holds with the reverse inequality, then we characterize the purely nonlinear part $\calS$ as an affine version of the geometric join of the quadrics $\partial_a D_i$.
The theory is developed classically in the projective setting, see, e.g., \cite[Chapter 1]{Russo:GeometrySpecialProjVars} and \cite{FleOCVog:JoinsAndIntersections}. In this section we translate some of these projective notions to the affine space and apply them to $\calS$.

Given two (complex) projective varieties $X,Y \subseteq \bbP^d$, their join is the projective variety $J(X,Y) = \bar{ \{ p \in \langle x,y \rangle : x \in X , y \in Y \}}$. We are concerned with the properties of the join summarized in the following lemma, which is a consequence of \cite[Example 18.17]{Harris:AlgGeo}.
\begin{lemma}\label{lemma: joins in general}
 Let $X,Y \subseteq \bbP^d$ be irreducible varieties. Then $J(X,Y)$ is irreducible. If  $X \cap Y = \emptyset$ then $\dim J(X,Y) = \dim X + \dim Y + 1$. Furthermore, if $\dim X + \dim Y + 1 < d$, then $\deg( J (X,Y)) = \deg(X) \deg(Y)$.
\end{lemma}

We prove an affine version of Lemma \ref{lemma: joins in general}, which will be useful to prove Theorem \ref{thm: joins} below. Regard the affine space $\bbC^d$ as an affine open subset of $\bbP^d$: for an affine variety $X \subseteq \bbC^d$, write $\bar{X} \subseteq \bbP^d$ for its Zariski closure and $X_\infty = \bar{X} \setminus X$ for its \emph{hyperplane cut at infinity}. Given two affine varieties $X,Y \subseteq \bbC^d$, we say that $X,Y$ do not intersect at infinity if $X_\infty \cap Y_\infty = \emptyset$. For two varieties $X,Y \subseteq \bbC^d$, write $\Sigma(X \times Y) = \bar{X + Y}$ for the Zariski closure of their Minkowski sum: this can be regarded as an affine version of the geometric join.

\begin{proposition}\label{prop: affine joins in general}
 Let $X,Y  \subseteq \bbC^d$ be irreducible affine varieties, not intersecting at infinity and such that $\dim X + \dim Y < d$. Then $\Sigma(X \times Y)$ is irreducible, $\dim \Sigma(X \times Y) = \dim X + \dim Y$ and $\deg \Sigma(X \times Y) = \deg(X) \deg(Y)$.
\end{proposition}
\begin{proof}
 The variety $\Sigma(X \times Y)$ is the closure of the image of the addition map $\Sigma : X \times Y \to \bbC^d$ defined by $\Sigma(x,y) = x+y$. Since $X,Y$ are irreducible, $\Sigma(X \times Y)$ is irreducible as well.
 
 Let $Z, Z'$ be two $1$-dimensional vector spaces and consider $\bbP^{d+1} = \bbP (\bbC^d \oplus Z \oplus Z')$ with homogeneous coordinates $x_1 \vvirg x_d, z, z'$. Reembed $X,Y$ in $\bbP^{d+1}$ as follows:
\[
\begin{array}{llcll}
  X &\to \bbP^{d+1} & \qquad \qquad &  Y &\to \bbP^{d+1} \\ 
  x &\mapsto (x,1,0)& \qquad \qquad &  y &\mapsto (y,0,1); 
\end{array}
\]
denote by $\bar{X},\bar{Y}$ the closures (in the Zariski topology of $\bbP^{d+1}$) of the two images.

Observe that $\bar{X},\bar{Y}$ are disjoint. Indeed, if $p \in \bar{X} \cap \bar{Y}$, then in coordinates one has $z(p) = z'(p) = 0$; hence $p$ belongs to the intersection $X_\infty \cap Y_\infty$ of the two hyperplane  cuts at infinity, which is empty by hypothesis. Therefore $\bar{X} \cap \bar{Y} = \emptyset$.
By Lemma \ref{lemma: joins in general},  $J(\bar{X},\bar{Y})$ is irreducible, with $\dim J(\bar{X},\bar{Y}) = \dim \bar{X} + \dim \bar{Y} + 1$. Moreover, since $\dim \bar{X} + \dim \bar{Y} + 1 < d+1$, we obtain $\deg J (\bar{X},\bar{Y}) = \deg(\bar{X}) \deg(\bar{Y})$.

Now, one can check explicitly in coordinates that
\[
 \Sigma(X \times Y) = J(\bar{X}, \bar{Y}) \cap \{ z = z'  \neq 0\};
\]
in other words, $ \Sigma(X \times Y)$ is an affine chart of the hyperplane section $\{z = z'\}$ of $J(\bar{X}, \bar{Y})$. 
Note that $J(\bar{X}, \bar{Y}) \cap  \{ z = z'\}$ is irreducible. To see this, observe that in the affine chart $\{z \neq 0\}$ it coincides with $\Sigma(X \times Y) $ which is irreducible; therefore other irreducible components would be supported at $z = z' = 0$. However, there is no line $L = \langle x , y \rangle$ with $x \in \bar{X}$ and $y \in \bar{Y}$ such that $L \cap \{ z = z' = 0\} \neq \emptyset$, unless $x \in X_\infty$ or $y \in Y_\infty$. This shows that $J(\bar{X},\bar{Y}) \cap \{ z = z' = 0\} = J(X_\infty , Y_\infty)$; since 
\[
\dim J(X_\infty,Y_\infty) \leq  \dim X_\infty + \dim Y_\infty +1 = \dim J(\bar{X},\bar{Y}) - 2,
\]
$J(X_\infty,Y_\infty)$ is not an irreducible component of a hyperplane section of $J(\bar{X},\bar{Y})$. This proves that $J(\bar{X},\bar{Y})\cap \{z = z'\}$ is irreducible, hence its affine chart on $\{ z = z' \neq 0\}$ is irreducible as well.
We conclude
\begin{align*}
\dim \Sigma(X \times Y) &= \dim J (\bar{X},\bar{Y}) - 1 = \dim X + \dim Y ,\\
\deg \Sigma(X \times Y) &= \deg J (\bar{X},\bar{Y}) = \deg(X) \deg(Y).
\end{align*}
\end{proof}

Applying Proposition \ref{prop: affine joins in general} iteratively to the boundaries of the discs defining the discotope, we obtain the following result.
\begin{theorem}\label{thm: joins}
Let $\bfN = (0,N_2 \vvirg N_{d})\subseteq \bbN^d$ be such that $\sum_{m=1}^{d} (m-1) N_m  \leq d-1$. Let $\calD$ be a generic discotope in $\bbR^d$ of type $\bfN$. Then $\calS$ is irreducible of degree $2^N$, where $N = \sum N_m$.
\end{theorem}
\begin{proof}
Let $D_1 \vvirg D_N$ be the discs defining the discotope and let $d_i = \dim \langle D_i \rangle$; in particular $\dim \partial_a D_i = d_i-1$. For $n = 1 \vvirg N$, let 
\[
X_n = \Sigma \left( \prod_{i=1}^n \partial_a D_i\right). 
\]

First notice $X_N = \calS$. The inclusion $\calS \subseteq X_N$ is clear by the definition of $\calS$. For the other inclusion, we show that there is a (real) Euclidean open subset $U \subseteq \prod \partial D_i$ such that $\Sigma(U) \subseteq \calS$; passing to the Zariski closure we obtain the equality. Let $\bfxi = (\xi_1 \vvirg \xi_N) \in \prod \partial D_i$, and for every $i$ write $T_{\xi_i} \partial D_i$ for the (real) tangent space at $\xi_i$; note $\dim T_{\xi_i} \partial  D_i = d_i -1$, hence $\langle T_{\xi_i} \partial D_i : i = 1 \vvirg N\rangle$ is a proper linear subspace of $\bbR^d$. Let $u \in \bbR^d$ be a unit vector such that the hyperplane $u^\perp$ contains $\langle T_{\xi_i} \partial D_i : i = 1 \vvirg N\rangle$. Up to replacing $\xi_i$ with $-\xi_i$, assume $\lprod u,\xi_i\rprod \geq 0$. Let $p = \Sigma(\bfxi) = \xi_1 + \cdots + \xi_N$. By definition $p \in X_N$; moreover $p \in \partial \calD$, because 
\[
 \lprod u , p \rprod = \textsum _1^N \lprod u , \xi_i\rprod \geq \textsum_1^N \lprod u , \tilde{\xi}_i \rprod = \lprod u , \tilde{p}\rprod
\]
for any other point $\tilde{p} = \tilde{\xi}_1 + \cdots + \tilde{\xi}_N$ of $\calD$; this shows that $p$ belongs to the face of $\calD$ exposed by $u$, and in particular to the boundary of $\calD$. Therefore $p \in \calS$. We conclude $X_N = \calS$.

Next, we show that for every $n$, $X_{n-1}$ and $\partial_a D_n$ have no intersection at infinity, for a generic choice of the discs. Having empty intersection at infinity is an open condition on the parameter space of the embeddings of the discs; hence, in order to show that there is no intersection at infinity for a generic choice of embeddings, it suffices to exhibit a choice for which this property is verified. 

By assumption, $\sum_1^N (d_i - 1) \leq d-1$; let $\delta = \max\{ 0,\big(\sum_1^N d_i\big) -d\}$ and notice $\delta \leq N-1$. Then, one can choose the embeddings of $D_1 \vvirg D_N$ so that the following properties hold:
\begin{itemize}
\item[\scalebox{.8}{$\blacktriangleright$}] if $n = 1 \vvirg \delta+1$, then the dimension of $\langle D_1 \vvirg D_{n-1} \rangle \cap \langle D_n \rangle$ coincides with $\dim (\langle D_{n-1} \rangle \cap \langle D_n \rangle)$ and it is exactly $1$,
\item[\scalebox{.8}{$\blacktriangleright$}] if $n = \delta+2 \vvirg N$, then $\langle D_1 \vvirg D_{n-1} \rangle \cap  \langle D_{n} \rangle= 0$.
\end{itemize}
With this choice of embeddings, we show that for every $n$, $X_{n-1}$ and $\partial_a D_n$ have no intersection at infinity. Write $X_{n-1 ,\infty}$ and $\partial_a D_{n ,\infty}$ for the two components at infinity; they are subvarieties of $\bbP (\langle D_1 \vvirg D_{n}\rangle)$. Their intersection is a subvariety of $\bbP (\langle D_1 \vvirg D_{n-1} \rangle )  \cap \bbP (\langle D_{n} \rangle)= \bbP \left( \langle D_{n-1} \rangle\cap \langle D_{n} \rangle \right)$. If $n \leq \delta+1$,   $\bbP \left( \langle D_{n-1} \rangle\cap \langle D_{n} \rangle \right)$ is a single point, and such point does not belong to $\partial_a D_{n, \infty}$; if $n \geq \delta +2$, then $\bbP \left( \langle D_{n-1} \rangle\cap \langle D_{n} \rangle \right)$ is empty. This proves the claim.

To conclude, we use induction on $n$ to show that $\dim X_n = \sum_{i=1}^n (d_i-1)$ and $\deg X_n = 2^n$. The statement is clear for $n =1$. Assume $n\geq 2$. We have $X_n = \Sigma(X_{n-1} \times \partial_a D_n)$. Since $X_{n-1}$ and $\partial_a D_n$ do not intersect at infinity, Proposition \ref{prop: affine joins in general} applies, hence 
\begin{align*}
 \dim X_n &= \dim X_{n-1} + \dim \partial_a D_n = \sum_{i=1}^{n-1} (d_i-1) + (d_n - 1) = \sum_{i=1}^n (d_i-1), \\
 \deg X_n &= \deg X_{n-1} \cdot \deg \partial_a D_n = 2^{n-1} \cdot 2  = 2^n. 
\end{align*}
For $n = N$, we obtain the desired result for $\calS$.
\end{proof}

\begin{example}
Consider the following discs in $\bbR^6$:
\begin{align}
D_1 &= \{(x_1\vvirg x_6) : x_3 = x_4 = x_5 = x_6 = 0, x_1^2+x_2^2\leq 1\},\\
D_2 &= \{(x_1\vvirg x_6) : x_1 = x_2 = x_5 = x_6 = 0, x_3^2+x_4^2\leq 1\},\\
D_3 &= \{(x_1\vvirg x_6) : x_1-x_3 = x_2 = x_4 = 0, (x_1+x_3)^2+x_5^2+x_6^2\leq 1\}.
\end{align}
Let $\calD = D_1 + D_2 + D_3$. This discotope is full dimensional but the condition \eqref{eq:condition_S} holds with reverse inequality: indeed, $1+1+2<5$. Thus, by Theorem \ref{thm: joins}, $\calS = \overline{\partial_a D_1+\partial_a D_2+\partial_a D_3}$ is irreducible, of codimension $2$ and degree $8$. Its ideal is generated by one cubic and three quartic polynomials:
\[
\begin{smallmatrix}
4 x_1^2 x_3+4 x_2^2 x_3-4 x_1 x_3^2-4 x_1 x_4^2+x_1 x_5^2-x_3 x_5^2+x_1 x_6^2-x_3 x_6^2+3 x_1-3 x_3, \\
~\\
~\\
16 x_3^4+32 x_3^2 x_4^2+16 x_4^4+8 x_3^2 x_5^2-8 x_4^2 x_5^2+x_5^4+8 x_3^2 x_6^2-8 x_4^2 x_6^2+2 x_5^2 x_6^2+x_6^4-40 x_3^2-24 x_4^2+6 x_5^2+6 x_6^2+9 ,\\
~ \\
~\\
16 x_1^4+32 x_1^2 x_2^2+16 x_2^4+8 x_1^2 x_5^2-8 x_2^2 x_5^2+x_5^4+8 x_1^2 x_6^2-8 x_2^2 x_6^2+2 x_5^2 x_6^2+x_6^4-40 x_1^2-24 x_2^2+6 x_5^2+6 x_6^2+9 ,\\
~\\
~\\
16 x_1^2 x_3^2+16 x_2^2 x_3^2+16 x_1^2 x_4^2+16 x_2^2 x_4^2  -4 x_1^2 x_5^2 -4 x_2^2 x_5^2-4 x_3^2 x_5^2-4 x_4^2 x_5^2-4 x_1^2  x_6^2-4 x_2^2 x_6^2-4 x_3^2 x_6^2-4 x_4^2 x_6^2+ \\
x_5^4 +2 x_5^2  x_6^2 +x_6^4+ 16 x_1 x_3 x_5^2+16 x_1 x_3 x_6^2-12 x_1^2-12 x_2^2-16 x_1 x_3-12 x_3^2-12 x_4^2+6x_5^2+6x_6^2+9 .
\end{smallmatrix}
 \]
\end{example}

\section{Exposed points of the discotope}\label{sec:exposed_points}
In the rest of the paper, we assume that \eqref{eq:condition_S} is satisfied. Recall that all extreme points of $\calD$, hence all its exposed points, are contained in $\calS$. In this section, we prove that they form a full dimensional subset of the boundary of the discotope; in particular, at least one irreducible component of $\calS$ of dimension $d-1$ is the Zariski closure of a subset of exposed points. Further, we prove that exposed points are generically exposed by a unique vector in $S^{d-1}$. First, we give a general result which will be useful in the following.

\begin{lemma}\label{lem:Minkowski_sum}
Let $K_1 \vvirg K_N$ be convex bodies in $\bbR^d$. Consider a point $p = p_1 + \ldots + p_N \in \partial K$ where $K$ is the Minkowski sum of the $K_i$'s, and assume that $p_i$ is a smooth point of $\partial K_i$ for every $i=1\vvirg N$. Fix $u\in S^{d-1}$. Then $T_{p_i} \partial K_i \subseteq u^\perp$ for every $i$ if and only if $p$ belongs to the face of $K$ exposed by $u$.
\end{lemma}
\begin{proof}
Assume $T_{p_i} \partial K_i \subseteq u^\perp$ for every $i$ for some $u\in S^{d-1}$. Then one of these vectors $u$ satisfies $p_i \in K_i^u$ for every $i$. As a consequence, $p\in K^u$.
Conversely, let $p\in K^u$. Therefore, $h_K(u) = \lprod p,u\rprod = \sum_{i=1}^N \lprod p_i ,u\rprod$ and for every $i$ and every $\tilde{p}_i \in D_i$, 
\begin{equation}
\lprod p_i ,u\rprod \geq \lprod \tilde{p}_i ,u\rprod.
\end{equation}
Hence $h_{D_i} (u) = \lprod p_i , u \rprod$. There are two possible situations: either $u \perp \langle D_i \rangle$, or $p_i$ is exposed by $u$. In both cases it is clear that $T_{p_i} \partial K_i \subseteq u^\perp$.
\end{proof}

\begin{proposition}
Let $\Sigma : \prod \partial_a D_i \to \bbC^d$ be the restriction of the addition map to the algebraic boundaries of $N$ generic discs. Assume that \eqref{eq:condition_S} holds with strict inequality. Then 
\begin{equation}
 \Sigma^{-1} (\calS) \subseteq \crit \left( \Sigma\right).
\end{equation}
Here $\crit \left( \Sigma\right)$ denotes the critical locus, that is the set of points $\bfxi \in \prod \partial_a D_i$ where the differential $d_{\bfxi}\Sigma$ does not have full rank.
\end{proposition}
\begin{proof}
Since \eqref{eq:condition_S} holds with strict inequality, for a generic $\bfxi \in \prod \partial_a D_i$ the differential $d_{\bfxi}\Sigma$ is surjective.
By density, it is enough to check that $d_{\bfxi}\Sigma$ is not surjective at the real points of $\Sigma^{-1} (\calS)$. For every $\bfxi\in \prod \partial D_i$, the image of the differential $d_{\bfxi} \Sigma$ is the sum $T_{\xi_1} \partial D_1 + \cdots + T_{\xi_N} \partial D_N$. If $\Sigma (\bfxi)$ belongs to the face $\calD^u$, then by Lemma \ref{lem:Minkowski_sum} $T_{\xi_i} \partial D_i \subseteq u^\perp$ for every $i$. In particular the differential is not surjective, hence $\bfxi$ is a critical point of $\Sigma$. Passing to the Zariski closure, we obtain $ \Sigma^{-1} (\calS) \subseteq \crit \left( \Sigma\right)$.
\end{proof}

The next result identifies a region of $\partial \calD$ of points exposed by a unique vector of $S^{d-1}$. 

\begin{lemma}\label{lem:omega}
Let $\calD$ be a generic discotope such that condition \eqref{eq:condition_S} is satisfied. Let $p \in \calD^{\partial} \cap \partial \calD$. The following are equivalent:
\begin{itemize}
\item there exists a unique $u \in S^{d-1}$ such that $p \in \calD^u$; 
\item $p = \textsum_{i=1}^N \xi_i$ for some $\xi_i \in \partial D_i$ such that $\codim \left( \textsum_{i = 1}^N T_{\xi_i} \partial D_i \right) = 1$.
\end{itemize}
Let $\Omega$ be the set of points that satisfy either (hence both) these conditions; then $\Omega$ is non-empty and Euclidean open in $\calD^{\partial} \cap \partial \calD$.
\end{lemma}
\begin{proof}
The equivalence of the two conditions follows from Lemma \ref{lem:Minkowski_sum}.
To show that $\Omega$ is non-empty, we construct a point in the following way. Consider $u\in\calU$ and let $p =  \sum_{i=1}^N \xi_i = \calD^u$. For the sake of notation, write $T_{\xi_i} = T_{\xi_i} \partial D_i$. Suppose that $L_{\bfxi} = T_{\xi_1} + \ldots + T_{\xi_N}$ is a subspace of codimension $c\geq 2$. Since $u\in\calU$, for every $i$ we have $\langle D_i \rangle \not \subseteq L_{\bfxi}$. Condition \eqref{eq:condition_S} implies that, up to relabeling, $T_{\xi_1} \cap \left( T_{\xi_2} + \ldots + T_{\xi_N} \right) = L' \neq \{0\}$. Let $L''$ be a complement of $L'$ in $T_{\xi_1}$, so that 
\begin{equation}
L' + L'' = T_{\xi_1} \qquad \hbox{and} \qquad L' \cap L'' = \{0\}.
\end{equation}
Consider the set of points $\bar{\xi}_1 \in \partial D_1$ such that $T_{\bar{\xi}_1}\supseteq L''$; let $\bar{\bfxi} = (\bar{\xi}_1, \xi_2 \vvirg \xi_N)$. For a generic choice of such $\bar{\xi}_1$ there exists $\bar{u}\in\calU$ such that $L_{\bar{\bfxi}} \subseteq \bar{u}^\perp$. Therefore the point $\bar{p} = \bar{\xi}_1 + \xi_2 + \ldots + \xi_N$ is an exposed point of $\calD$.  Moreover, if $\bar{\xi}_1\neq \pm \xi_1$ then $\codim L_{\bar{\bfxi}} \leq c -1$. 
Repeating this argument one constructs a point $\bfxi$ such that $\codim (T_{\xi_1} + \ldots + T_{\xi_N}) = 1$.
The condition that this codimension is $1$ is Zariski open, hence $\Omega$ is Euclidean open in $\calD^{\partial} \cap \partial \calD$.
\end{proof}

Recall the following property of the support function of a convex body $K$, see \cite[Corollary 1.7.3]{Sch:ConvexBodies}. The support function is differentiable at $u\in S^{d-1}$ if and only if the face $K^u$ is a unique point; this point coincides with $\nabla h_K(u)$. In particular, $h_{\calD}$ is differentiable at all points of $\calU$. This will be useful in the next result, to prove that the set of exposed points of $\calD$ is full dimensional in its boundary.

\begin{proposition}\label{prop:vertices}
In the hypotheses of Lemma \ref{lem:omega}, there exists an open dense subset $\calU'$ of $\calU$ such that $\nabla h_\calD| _{\calU'}$ is one to one.
\end{proposition}

\begin{proof}
Fix $u\in \calU$ and denote by $p_u$ the point of the discotope exposed by $u$. Let $\xi_i \in \partial D_i$ be the unique point of the $i$-th disc such that
$ h_{D_i}(u) = \lprod \xi_i,u \rprod$; then $p_u = \sum_{i=1}^N \xi_i$. The tangent space $T_{\xi_i} \partial D_i = u^\perp \cap \langle \calB_i \rangle$  is a $(\dim D_i -1)$-dimensional subspace of $u^\perp$. Because of the non-degeneracy condition \eqref{eq:condition_S}, there exists a Euclidean open and dense subset $\calU'$ of $\calU$ such that for all $u\in\calU'$
\begin{equation}\label{eq:tangent_spaces_U'}
\sum_{i=1}^{N} T_{\xi_i} \partial D_i = u^\perp.
\end{equation}
If $u\in \calU'$ then the support function of the discotope is smooth in a neighborhood of $u$, because $h_\calD (u) = \sum h_{D_i} (u)$ and the $h_{D_i}$'s are smooth in a neighborhood of $u$. Hence we have the following map
\begin{align}
\nabla h_\calD| _{\calU'} : \calU' &\to \partial \calD \\
u &\mapsto \nabla h_\calD (u) = p_u.
\end{align}
Its image lies inside $\Omega$ because of \eqref{eq:tangent_spaces_U'}. Since these are exactly the points exposed by only one direction, $\nabla h_\calD| _{\calU'}$ is one to one.
\end{proof}

From Proposition \ref{prop:vertices}, we see that $\nabla h_\calD(\calU')$ is a set of exposed points which is open in $\partial \calD$; in particular it has dimension $d-1$. Moreover, $\nabla h_\calD$ defines a diffeomorphism between $\calU'$ and its image, therefore $\nabla h_\calD(\calU')$ consists of smooth points of $\partial_a \calD$. A consequence of this is that the Zariski closure of the exposed points (or equivalently of the extreme points) contains at least one irreducible component of $\calS$ of dimension $d-1$. This leads to the following result.
\begin{corollary}\label{cor:dim_S}
Let $\calD$ be a generic discotope such that the non-degeneracy condition \eqref{eq:condition_S} holds. Then $\calS$ has at least one irreducible component of dimension $d-1$ and this is an irreducible component of the algebraic boundary $\partial_a \calD$.
\end{corollary}
In general, it is not clear whether $\calS$ has multiple irreducible components, possibly even of different dimension. Indeed, the set $\calD^\partial = \Sigma \big( \textprod_{j} \partial D_{j} \big)$, introduced in Definition \ref{def:S}, may intersect positive dimensional faces of $\calD$. This might produce lower dimensional components of $\calS$. However, we expect this not to be the case, as stated in Conjecture \ref{conj:irr_S}.

We conclude this section pointing out that in general some boundary points of $\calD$ can be exposed by more than one vector. These can be identified by the following condition. Set $L_i = \langle D_i \rangle$, so that $L_1 \vvirg L_N$ are $N$ generic linear subspaces of $\bbR^d$. Consider the hyperplanes $H = u^\perp\subseteq \bbR^d$ for $u\in \calU$; hence $\dim \left( H\cap L_i \right) = \dim L_i - 1$ for every $i$. A point $p = \calD^u$ is exposed by more than one vector if and only if $H$ satisfies
\begin{equation}\label{eq:polymatroid}
\dim \left( \left( H \cap L_1 \right) + \ldots + \left( H \cap L_N \right) \right) \leq d-2.
\end{equation}
Indeed, when \eqref{eq:polymatroid} holds, there exists a linear subspace $V$ of dimension at least one such that $H = \left( H \cap L_1 \right) + \ldots + \left( H \cap L_N \right) +V$. By perturbing $V$ we obtain a family of hyperplanes that expose the point $p$.
The condition \eqref{eq:polymatroid} can be formulated in terms of a degeneracy property of an associated polymatroid, but a full characterization seems difficult in general. In the case $N=2$, boundary points exposed by more than one vector always occur; the hyperplanes exposing them are characterized in the following example.

\begin{example}\label{ex:two_planes_vertices}
Let $d_1 = \dim D_1$, $d_2 = \dim D_2$. By the non-degeneracy condition \eqref{eq:condition_S}, $d_1 + d_2 \geq d+1$. Let $L_i = \langle D_i\rangle$; by genericity $\dim (L_1 \cap L_2) = d_1+d_2-d$. For a hyperplane $H$ such that $\dim (L_i \cap H ) = d_i-1$, we have 
\begin{align*}
 &\dim ((L_1 \cap H) + (L_2 + H)) = \\ 
 &(d_1 - 1) + (d_2 - 1) - \dim (L_1 \cap L_2 \cap H) = \left\{ \begin{array}{ll} d - 2 & \text{ if $L_1 \cap L_2 \subseteq H$,}  \\ 
 d-1 & \text{otherwise.}
\end{array}
\right. 
\end{align*}
Therefore, a point $p \in \calD$ exposed by a hyperplane $H$ such that $L_1 \cap L_2 \not\subseteq H$, is exposed only by such hyperplane. On the other hand, if $p$ is exposed by a hyperplane $H$ with $L_1 \cap L_2 \subseteq H$, then there is a cone of hyperplanes $\tilde{H}$ with $L_1 \cap L_2 \subseteq \tilde{H}$ exposing $p$ as well. \\
The case $d = 3$, $d_1 = d_2 = 2$ is shown in Figure \ref{fig:space_discotopes}, right. This discotope $\calD'$ is defined in Example \ref{ex:discotopes_segment}; in this case $L_1\cap L_2$ is the vertical $x_3$-axis. The plane $\{x_3 = 0\}$ is partitioned into four $2$-dimensional cones and every $u$ in the interior of the same cone exposes the same point. These four exposed points are the pairwise intersection of two adjacent blue discs.
\end{example}

\section{Two-dimensional discs in $\bbR^d$}\label{sec:2discs}

In this section, we consider discotopes $\calD\subseteq \bbR^d$ of type $(0,N,0\vvirg 0)\in \bbN^{d}$, that are realized as sum of $2$-dimensional discs. If $N\leq d-1$, the variety $\calS$ is described by Theorem \ref{thm: joins}. Thus assume that $N\geq d-1$, which ensures that $\dim \calS = d-1$.
We will prove that the purely nonlinear part $\calS$ is irreducible, hence it is the Zariski closure of the extreme points of $\calD$. In addition, we will provide an upper bound for the degree of this component of the algebraic boundary. 
\begin{theorem}\label{thm:degS}
Let $\calD$ be a generic discotope of type $(0 , N , 0 \vvirg 0)$ in $\bbR^d$, with $N \geq d-1$. Let $\calS$ be the purely nonlinear part of $\calD$. Then $\calS$ is irreducible, and coincides with the Zariski closure of the set of extreme points of $\calD$. Moreover,
\begin{equation}
 \deg(\calS) \leq 2^N \cdot \binom{N}{d-1}.
\end{equation}
\end{theorem}
Let $D_1 \vvirg D_N$ be $2$-dimensional discs in $\bbR^d$ in general position. For every $j=1\vvirg N$, consider the (complexification of the) embedding $A_j : \bbC^2 \to \bbC^d$ defining the generalized disc $D_j$; let $\calB_j = \{ b^{(j)}_1 , b^{(j)}_2 \}$ be the associated basis of the image of $A_j$. Then the product $\prod_{j=1}^N \partial_a D_j$ is the image of the restriction of $A = A_1 \ttimes A_N$ to $\prod \{c_j^2 + s_j^2=1\} \subseteq (\bbC^2)^N$. Here $(c_j,s_j)$ are the coordinates on the $j$-th copy of $\bbC^2$. Consider the addition map
\[
 \Sigma : \partial_a D_1 \ttimes \partial_a D_N \to \bbC^d.
\]
The critical locus of the restriction of $\Sigma \circ A$ is the variety defined by the ideal
\begin{equation}\label{eq:gensI}
    I = \Delta + \big( c_1^2+s_1^2-1, \ldots, c_N^2+s_N^2-1 \big) \subseteq \bbC[c_1,s_1 \vvirg c_N,s_N],
\end{equation}
where $\Delta$ is the ideal of the $d\times d$ minors of the $N\times d$ matrix
\begin{equation}\label{matrix:cs}
M = 
\begin{pmatrix}
b_2^{(1)} c_1 - b_1^{(1)} s_1 \\
\vdots \\
b_2^{(N)} c_N - b_1^{(N)} s_N 
\end{pmatrix}.
\end{equation}
This is the (transpose of the) matrix representing the differential of the restriction of $\Sigma \circ A$.
Since $A$ is a linear embedding, $\crit \Sigma $ is irreducible if and only if $\crit (\Sigma \circ A)$ is irreducible, and their degrees coincide. 

We will prove the irreducibility of $\crit \Sigma$ and compute its degree by first studying the variety $\calV(\Delta)$. We show that it is irreducible and that its degree coincides with the one of the classical determinantal variety of $N\times d$ matrices of submaximal rank. This is a consequence of Lemma \ref{lemma: linear sections determinantal}, which provides a more general result on special linear sections of the determinantal variety. This topic is object of classical study, see \cite{Eis:LinSec}, \cite[Section 6B]{Eisenbud:SyzygyBook}. However, these results rely on a specific condition, called $1$-genericity, which is not satisfied in our setting.

We state the following version of Bertini's Theorem for projective varieties, which can be obtained from \cite[Theorem 6.3]{jouanolouBertini} applied to the special case of quasi-projective varieties.
\begin{lemma}\label{thm: bertini}
 Let $X$ be an irreducible projective variety. Let $\calL$ be a line bundle on $X$ defining a map $\Phi:X \dashto \bbP^{h_0(\calL)-1}$ such that $\dim \Phi(X)\geq s+1$. Let $D_1 \vvirg D_s \in |\calL|$ be generic elements of the linear system defined by $\calL$. Let $Y = D_1 \cap \cdots \cap D_s$ and let $B$ be the base locus of $\calL$. Then $\bar{Y \setminus B}$ is irreducible of codimension $s$ in $X$.
\end{lemma}
\begin{proof}
The proof follows from \cite[Theorem 6.3 (4)]{jouanolouBertini} applied to the quasi-projective variety $\tilde{X} = X \setminus B$ and the morphism $\Phi|_{\tilde{X}}$.
\end{proof}

Informally, this result guarantees that the intersection of generic divisors is irreducible and of the expected codimension outside of the base locus of the line bundle.

We use Lemma \ref{thm: bertini} to prove that certain non-generic linear sections of the classical determinantal variety are irreducible and of the expected dimension. Let $\Mat_{n\times m}$ denote the (complex) vector space of $n \times m$ matrices and let 
\[
\calM^{n\times m}_r  = \{ A \in \bbP \Mat_{n \times m} : \rank(A) \leq r\}
\]
be the $r$-th determinantal variety. Use coordinates $x_{ij}$ on $\Mat_{n \times m}$, where $x_{ij}$ is the entry at row $i$ and column $j$.

\begin{lemma}\label{lemma: linear sections determinantal}
Let $m,n , r \geq 2$ be integers with $r < m,n$. Let $s$ be an integer $1 \leq s < r$. For $i = 1 \vvirg n$, let $\ell^{(i)}_1 \vvirg \ell^{(i)}_s$ be generic linear forms on $\Mat_{n\times m}$ only involving the variables $\{ x_{ij} : j = 1 \vvirg m\}$ of the $i$-th row. Let 
\[
 \calY^{n \times m}_r = \calM^{n \times m}_r \cap \left\{ A \in \bbP \Mat_{n\times m}: \ell^{(i)}_p(A) = 0 \text{ for } i = 1 \vvirg n, p = 1 \vvirg s \right\}.
\]
Then $ \calY^{n \times m}_r$ is irreducible and of codimension $ns$ in $\calM^{n\times m}_r$.
\end{lemma}
\begin{proof}
For $i = 1 \vvirg n$, let 
\[
 \Gamma_i = \{ A \in \bbP \Mat_{n \times m}: a_{ij} = 0 \text{ for all } j = 1 \vvirg m\}
\]
be the linear subspace of matrices having zero $i$-th row. Let $\Gamma = \bigcup_{i =1}^n \Gamma_i$.

For $t = 0 \vvirg n$, let 
\[
\calY^{(t)} = \calM^{n \times m}_r \cap \left\{ A \in \bbP \Mat_{n\times m}: \ell^{(i)}_p(A) = 0 \text{ for } i = 1 \vvirg t, p = 1 \vvirg s \right\} ;
\]
we have $\calM^{n \times m}_r = \calY^{(0)} \supseteq \calY^{(1)} \supseteq \cdots \supseteq \calY^{(n)} = \calY^{n\times m}_r$.

Let $\Phi_i : \bbP \Mat_{n \times m} \dashto \bbP ^{m-1}$ be the projection on the $i$-th row; $\Phi_i$ is a rational map, whose indeterminacy locus is $\Gamma_i$. Let $\calL_i = \Phi_i^* \calO(1)$ be the pullback of the hyperplane bundle on $\bbP^{m-1}$: global sections of $\calL_i$ are linear forms only involving the variables of the $i$-th row; in particular the base locus of $\calL_i$ is exactly $\Gamma_i$. 

For a fixed $n$, we use induction on $t$ to show that $\calY^{(t)}$ is irreducible up to components contained in $\Gamma$, in the sense that $\calY^{(t)} \setminus \Gamma$ is irreducible. 

If $t = 0$, then $\calY^{(t)} = \calM^{n \times m}_r$ is irreducible. If $t \geq 1$, then $\calY^{(t)}$ is the intersection of $s$ divisors $D_1 \vvirg D_s \in \bigl|\calL_i |_{\calY^{(t-1)}}\bigr|$ on $\calY^{(t-1)}$, where $D_p = \{ \ell^{(t)}_p = 0\}$ and $\calL_i |_{\calY^{(t-1)}}$ is the restriction of $\calL_i$ to $\calY^{(t-1)}$. By the induction hypothesis, $\calY^{(t-1)}$ is a union of irreducible components, only one of which is not contained in $\Gamma$. 

In order to apply Lemma \ref{thm: bertini}, we need $\dim \Phi_t \left( \calY^{(t-1)} \right) \geq s+1$. In fact we show that $\Phi_t \left( \calY^{(t-1)} \right) = \bbP^{m-1}$. If $t\leq r$, this is clear because for every choice of the first $t$ rows, the corresponding matrix can be completed to a rank $r$ matrix. If $t>r$, notice that every $r$-dimensional subspace $E \subset \bbC^m$ can be realized as the span of the first $t-1$ rows of a matrix in $ \calY^{(t-1)} $. Since $s<r$, for every $i = 1 \vvirg t-1$ the intersection of $E$ with the subspace of $\bbC^m$ cut out by the linear forms $\ell^{(i)}_1 \vvirg \ell^{(i)}_s$ is non-trivial. Consider the matrix $A\in \calY^{(t-1)}$ whose $i$-th row is a generic element of this intersection for $i < t$, and suitably completed to a rank $r$ matrix. By the genericity of the linear forms the span of the first $t-1$ rows of $A$ is exactly $E$. Fix now $v \in \bbC^m$ and let $E$ be an $r$-dimensional subspace containing $v$. The associated matrix $A$ constructed above can be chosen so that the $t$-th row coincides with $v$. In this way, $\Phi_t (A) = v$ and $\dim \Phi_t \left( \calY^{(t-1)} \right) = m-1 \geq s+1$ follows.

Therefore Lemma \ref{thm: bertini} applies and we obtain that $\calY^{(t)}$ is irreducible up to components contained in the base locus of $\calL_i$, that is $\Gamma_i \subseteq \Gamma$. This proves the desired property for $\calY^{(t)}$ and, in particular, shows that $\calY^{n \times m}_r$ is irreducible up to components contained in $\Gamma$. 

For every $t$, let $Y^{(t)}$ be the component of $\calY^{(t)}$ not contained in $\Gamma$. In particular, $Y^{(t)}$ is not contained in the base locus of $\calL_i|_{Y^{(t-1)}}$; therefore, by Lemma \ref{thm: bertini}, it has the expected codimension. This provides $\codim _{\calM^{n\times m}_r} (Y^{(t)}) = ts$.

Finally, we prove that in fact $\calY^{n \times m}_r$ does not have components contained in $\Gamma$, thus it is irreducible. This is proved by induction on $n$. 
The base case of the induction is $n=r$. In this case $\calM^{n \times m}_r$ is the whole space $\bbP \Mat_{r \times m}$ and $\calY^{n \times m}_r$ is the transverse intersection of $ns$ linear spaces. Therefore it is irreducible.

Let $n > r$. Suppose by contradiction that $\calY^{n \times m}_r$ has at least one component, denoted by $C$, contained in $\Gamma$. Then $C \subseteq \Gamma_i$ for some $i$; without loss of generality, suppose $i=n$. Identify $\Mat_{(n-1) \times m}$ with the subspace of $\Mat_{n \times m}$ having the $n$-th row equal to $0$. Under this identification, the component $C$ is contained in $\calY^{(n-1) \times m}_r$, so $\dim C \leq \dim \calY^{(n-1) \times m}_r$. By the induction hypothesis, $\calY^{(n-1) \times m}_r$ is irreducible, so it coincides with its only component not contained in $\Gamma$ and in particular it has the expected codimension in $\calM^{(n-1) \times m}_r$. We obtain
\begin{align*}
 \dim C \leq  \dim \calY^{(n-1) \times m}_r &= \dim \calM ^{(n-1) \times m}_r - (n-1)s \\
  &= r((n-1)+m-r) - (n-1)s \\
  &= r(n+m-r) - ns - (r-s) \\
  &= \dim \calM^{ n \times m}_r - ns - (r-s).
\end{align*}
This implies $\codim _{\calM^{n \times m}_r}(C) > ns$ in contradiction with the fact that $\calY^{n \times m}_r$ is cut out by $ns$ equations in $\calM^{n \times m}_r$. We conclude that $\calY^{n \times m}_r$ has no components contained in $\Gamma$; thus it is irreducible. 
\end{proof}

\begin{remark}
In Lemma \ref{lemma: linear sections determinantal}, it is not necessary to have the same number of linear relations on every row. The same argument applies if, on the $i$-th row, one has $s_i$ linear relations, with $s_i < r$ for every $i$. Then $\calY^{n \times m}_r$ is irreducible and of codimension $\sum s_i$ in $\calM^{n\times m}_r$.
\end{remark}
Lemma \ref{lemma: linear sections determinantal} shows that linear sections of the determinantal variety only involving a single row are \emph{generic enough} in the sense that they preserve irreducibility and have the expected dimension. We apply Lemma \ref{lemma: linear sections determinantal} to the variety $\calV(\Delta)$: in this case $r = d-1$ and $s = d-2$.

\begin{proposition}\label{prop:degCrit}
The variety $\crit ( \Sigma )$ is irreducible, of dimension $d-1$, and degree $2^N\binom{N}{d-1}$.
\end{proposition}
\begin{proof}
Since $A = A_1 \ttimes A_N$ is a linear embedding, it suffices to prove the statement for $\crit (\Sigma \circ A)$, that is the variety defined by the ideal $I$ in \eqref{eq:gensI}.

By Lemma \ref{lemma: linear sections determinantal}, the variety $\calV(\Delta) \subseteq \bbC^2 \ttimes \bbC^2$ is irreducible of dimension $N + d-1$.  Consider its closure in $\bbP^2 \ttimes \bbP^2$, where the $j$-th copy of $\bbP^2$ has homogeneous coordinates $[c_j,s_j , z_j]$. For every $j = 1\vvirg N$, the polynomial $c_j^2 + s_j^2 - 1$ on $\bbC^2$ defines a homogeneous quadric $\{ c_j^2 + s_j^2 - z_j^2 = 0\}$ on $\bbP^2$. This gives a generic element of $| \calO_{\bbP^2}(2) |$, which pulls back to a generic element $ Q_j \in | \calO_{(\bbP^2)^{N}} ( 0 \vvirg 0,2,0 \vvirg 0 )|$.
Recursively applying Lemma \ref{thm: bertini}, for every $j$ we have that $\calV(\Delta) \cap Q_1 \cap \cdots \cap Q_j$ is irreducible of dimension $N-j+d-1$. For $j = N$, we obtain the irreducibility of $\calV(\Delta) \cap Q_1 \cap \cdots \cap Q_N$.

As a consequence, $\crit(\Sigma \circ A) = \calV(I)$ is irreducible of dimension $d-1$. In particular, the intersection of the determinantal variety $\calV(\Delta)$ with the quadrics is dimensionally transverse. Moreover, $\calV(\Delta)$ is arithmetically Cohen-Macaulay, see e.g. \cite[Chapter 2]{ArCoGrHa:Vol1}. Therefore \cite[Corollary 2.5]{EisHar:3264} guarantees 
\begin{align*}
 \deg \left(\crit \left( \Sigma \circ A \right) \right) &= \deg \Big(\calV\left( \Delta \right) \Big) \cdot \prod_{i=1}^N \deg \left( \partial_a D_i \right) = \binom{N}{d-1} \cdot 2^N.
\end{align*}
\end{proof}

\begin{proof}[Proof of Theorem \ref{thm:degS}]
The irreducibility of $\crit(\Sigma)$ implies the irreducibility of its image under the addition map $\Sigma$, that is the purely nonlinear part $\calS$. By the linearity of $\Sigma$, we obtain an upper bound on the degree of $\calS$:
\[
 \deg(\calS) \leq 2^N \cdot \binom{N}{d-1} .
\]
From the discussion in Section \ref{sec:exposed_points}, the set of extreme points of $\calD$ is contained in $\calS$ and contains a Zariski dense subset of (at least) one of the components of $\calS$. By irreducibility, we conclude. 
\end{proof}

We end this section with some observations in the case of discotopes of type $\bfN = (0,N)$ in $\bbR^2$. In this case $\partial_a\calD = \calS$, which is an irreducible curve of degree $2^N \cdot N$. 
The real points of $\crit \Sigma $ come naturally in $2^{N-1}$ connected components, described as follows. Given a line $\ell \subseteq \bbR^2$ through the origin, there are exactly two points $\pm p_i$ on each ellipse $\partial D_i$ such that $T_{\pm p_i} \partial D_i$ is parallel to $\ell$. The choice of these signs (up to a global sign) determines locally a parametrization of the real points of $\crit \Sigma$, which has $2^{N-1}$ connected components. After the projection to $\bbR^2$,  many components of the real points of $\crit \Sigma$ can be mapped to the same connected component of the real points of $\calS$. This can be visualized in the example in Figure \ref{fig:planar_discotopes}, where the red curve $\calS$ is union of $2^2 = 4$ subsets homeomorphic to circles: these are the images of the $4$ connected components of $\crit \Sigma$. Exactly one of them is the topological boundary of $\calD$.

Furthermore the degree of the map $\Sigma : \crit(\Sigma) \to \calS$ is odd. By a density argument, this can be computed considering the fiber over a generic point $p \in \partial \calD$. This contains a single real point $(\xi_1 \vvirg \xi_N)$ where $\xi_j \in \partial D_j$ is the unique point exposed by the vector $u \in S^{1}$ which exposes $p$; the non-real points of $\Sigma^{-1}(p)$ come in pairs of complex conjugates, therefore there is an even number of them. We conclude that the fiber $\Sigma^{-1}(p)$ consists of an odd number of points, hence the degree of $\Sigma$ is odd.

\begin{remark}\label{rmk:L4ball}
In the case $d=2$ the degree of the critical locus of $\Sigma$ is $2^N\cdot N$ and the degree of the map $\Sigma : \crit \Sigma \to \calS$ is odd. Write $N = 2^\kappa \cdot M$, with $M$ odd. Then $\deg(\calS)$ is necessarily an odd multiple of $2^{N} \cdot 2^\kappa$.
A consequence of this is that the unit ball of the $L^4$-norm $\{ x_1^4 + x_2^4 \leq 1 \}$ is not a discotope. If it was a discotope, it would be of type $(0,N)$ for some $N \geq 2$. But this discussion shows that no curve of degree $4$ is the boundary of a discotope of type $(0,N)$ in $\bbR^2$.
\end{remark}

\section{The Dice}\label{sec: dice}

In this section, we provide an extended analysis of the algebro-geometric features of the surface $\calS \subseteq \bbC^3$ for a specific discotope of type $\bfN = (0,3,0)$. 
In particular, we will show that $\calS$ is birational to a smooth K3 surface, realized as a divisor of multidegree $(2,2,2)$ in $\bbP^1 \times \bbP^1 \times \bbP^1$.
This example was first studied in \cite[Section 5.3]{MatMer:FiberConvexBodies} in the context of fiber bodies: it provides a semialgebraic convex body having a fiber body which is not semialgebraic. Notice that, up to changing coordinates, the generic case of a discotope of type $\bfN = (0,3,0)$ can be reduced to the case where the three generalized discs of interest lie in the three coordinate hyperplanes. We further restrict to the case of three unit discs: 
\begin{align*}
 D_1 = \{ (x_1,x_2,x_3) : x_1=0; x_2^2+x_3^2 \leq 1\} ,\\
 D_2 = \{ (x_1,x_2,x_3) : x_2=0; x_1^2+x_3^2 \leq 1\} ,\\
 D_3 = \{ (x_1,x_2,x_3) : x_3=0; x_1^2+x_2^2 \leq 1\} .
\end{align*}
Let $\calD \subseteq \bbR^3$ be the resulting discotope and let $\calS \subseteq \bbC^3$ be its purely nonlinear part. By \cite[Section 5.3]{MatMer:FiberConvexBodies} and Theorem \ref{thm:degS}, $\calS$ is an irreducible surface of degree $24$. Its defining polynomial $F_{\calS}$ has $455$ terms. Because of the symmetries of the problem, all the monomials appearing in $F_{\calS}$ are squares. Since $\calS$ is the image of a polynomial map, $F_{\calS}$ can be computed via elimination theory \cite[Section 4.4, Theorem 3]{CoxLitOSh:IdealsVarsAlgs}. More precisely, consider the ideal
\begin{equation}
J = I + \Big( (x_1,x_2,x_3) - \sum_{i=1}^3 b_1^{(i)}c_i + b_2^{(i)}s_i \Big) \subset \bbC[x_i,c_i,s_i : i = 1,2,3]
\end{equation}
where $b_1^{(1)} = b_2^{(3)} = (0,1,0)$, $b_1^{(2)} = b_2^{(1)} = (0,0,1)$, $b_1^{(3)} = b_2^{(2)} = (1,0,0)$ and $I$ is the ideal in \eqref{eq:gensI}. Then $F_{\calS}$ is the unique (up to scaling) generator of $J \cap \bbC[x_1,x_2,x_3]$ and it can be computed using a computer algebra software, e.g., \verb|Macaulay2| \cite{M2}.

One can verify that the surface $\calS$ is singular in codimension $1$. The singular locus is highly reducible and has degree $294$. Our next goal is to construct a desingularization of $\calS$. Consider the rational parametrization of the (complex) circle $\psi:  t  \mapsto  (\textfrac{1-t^2}{1+t^2}, \textfrac{2t}{1-t^2})$. Let $\Sigma \circ (\psi_1 \times \psi_2 \times \psi_3)$ be the composition of the addition map with the parameterization of the three circles $\partial_a D_i$;  explicitly
\begin{equation}\label{eqn: parameterization of S in dice}
\begin{array}{ccccc}
 \bbC \times \bbC \times \bbC &\overset{\psi_1 \times \psi_2 \times \psi_3}{\dashto} &\partial_{a} D_1 \times \partial_{a} D_2 \times \partial_{a} D_3& \to & \bbC^3 \\
 (t_1,t_2,t_3) & \mapsto & \left( \left(\begin{smallmatrix} 0 \\ \frac{1- t_1^2}{1+t_1^2} \\ \frac{2t_1}{1+t_1^2}\end{smallmatrix} \right) ,
			          \left(\begin{smallmatrix} \frac{2t_2}{1+t_2^2} \\ 0 \\ \frac{1- t_2^2}{1+t_2^2} \end{smallmatrix} \right) ,
				  \left(\begin{smallmatrix} \frac{1- t_3^2}{1+t_3^2} \\ \frac{2t_3}{1+t_3^2} \\ 0  \end{smallmatrix} \right)  \right) & & \\ ~ \\
& & \left( \left(\begin{smallmatrix} x_1 \\ x_2 \\ x_3 \end{smallmatrix} \right) ,
	 \left(	\begin{smallmatrix} y_1 \\ y_2 \\ y_3 \end{smallmatrix} \right) ,
	 \left(	\begin{smallmatrix} z_1 \\ z_2 \\ z_3 \end{smallmatrix}  \right) \right)  & \mapsto & \left(\begin{smallmatrix} x_1 +y_1+z_1\\ x_2 +y_2 + z_2\\ x_3 +  y_3 + z_3 \end{smallmatrix} \right).
\end{array} 
\end{equation}
The differential of the composition is 
\[
M(t_1,t_2,t_3) = \left( \begin{array}{ccc}
                                    0 & \frac{1-t_2^2}{1+t_2^2} & \frac{-2t_3}{1-t_3^2}  \\
                                    \frac{-2t_1}{1+t_1^2} & 0 & \frac{1-t_3^2}{1+t_3^2} \\
                                    \frac{1-t_1^2}{1+t_1^2} & \frac{-2t_2}{1+t_2^2} & 0 \\
                                   \end{array}
\right)
\]
so that the critical locus is the hypersurface in $\bbC^3$ determined by the vanishing of
\begin{equation}\label{eqn: critical locus phi affine}
\det( M(t_1,t_2,t_3) ) = \frac{1}{(1+t_1^2)(1+t_2^2) (1+t_3^2)} \left[(1-t_1^2)(1-t_2^2) (1-t_3^2) - 8t_1t_2t_3 \right].
\end{equation}
The map $\Sigma \circ (\psi_1 \times  \psi_2 \times \psi_3)$ extends to a regular map
$
 \phi: \bbP^1 \times \bbP^1 \times \bbP^1 \to \bbP^3;
$
from \eqref{eqn: critical locus phi affine}, we obtain that the critical locus of this extension is the surface $\tilde{\calS}$ of multidegree $(2,2,2)$ defined by the equation
\[
(s_1^2-t_1^2)(s^2_2-t_2^2) (s_3^2-t_3^2) - 8s_1s_2s_3t_1t_2t_3 = 0,
\]
where $[s_i,t_i]$ are homogeneous coordinates on the $i$-th copy of $\bbP^1$.

\begin{theorem}
 The surface $\tilde{\calS}$ is a smooth \textrm{K3} surface. The map $\phi$ is a birational equivalence between $\tilde{\calS}$ and $\calS$. In particular $\tilde{\calS}$ is a desingularization of $\calS$. 
\end{theorem}
\begin{proof}
The smoothness and the irreducibility of $\tilde{\calS}$ are verified by a direct calculation. 

It is a classical fact that a smooth divisor of multidegree $(2,2,2)$ in $\bbP^1 \times \bbP^1 \times \bbP^1$ is a K3 surface. For completeness, we give an explicit proof. Let $\calO_{\tilde{S}}$ and $\omega_{\tilde{\calS}}$ be the structure and the canonical sheaves of $\tilde{S}$, respectively. We verify the two conditions $ \omega_{\tilde{\calS}} \simeq \calO_{\tilde{S}}$ and $h^1(\calO_{\tilde{S}}) = 0$, characterizing a K3 surface.

First, we prove $\omega_{\tilde{\calS}} \simeq \calO_{\tilde{S}}$. This follows from the classical adjunction formula, see, e.g., \cite[Proposition 1.33]{EisHar:3264}. Since $\tilde{\calS}$ is a smooth divisor of multidegree $(2,2,2)$, we have
\[
 \omega_{\tilde{\calS}} = (\omega_{\bbP^1 \times \bbP^1 \times \bbP^1} ( 2,2,2) )\!|_{\tilde{\calS}} = (\calO_{\bbP^1 \times \bbP^1 \times \bbP^1} (-2,-2,-2)\otimes \calO_{\bbP^1 \times \bbP^1 \times \bbP^1} (2,2,2) ) \!|_{\tilde{\calS}}
\]
which is $\calO_{\tilde{\calS}}$. In order to show $h^1(\calO_{\tilde{\calS}}) = 0$, consider the restriction exact sequence of $\tilde{\calS}$:
\[
 0 \to \calI_{\tilde{\calS}} \to \calO_{\bbP^1 \times \bbP^1 \times \bbP^1} \to \calO_{\tilde{\calS}} \to 0.
\]
Again, since $\tilde{\calS}$ is a smooth divisor of multidegree $(2,2,2)$, we have $\calI_{\tilde{\calS}} \!\simeq\! \calO_{\bbP^1 \times \bbP^1 \times \bbP^1} ( -2,-2,-2)$; passing to the long exact sequence in cohomology, we have
\[
 \cdots \to H^1 ( \calO_{\bbP^1 \times \bbP^1 \times \bbP^1} ) \to H^1(  \calO_{\tilde{\calS}} ) \to H^2 ( \calO_{\bbP^1 \times \bbP^1 \times \bbP^1} ( -2,-2,-2) ) \to \cdots
\]
By K\"unneth's formula, $h^1 ( \calO_{\bbP^1 \times \bbP^1 \times \bbP^1} ) = 0$. Since $\calO_{\bbP^1 \times \bbP^1 \times \bbP^1} ( -2,-2,-2) = \omega_{\bbP^1 \times \bbP^1 \times \bbP^1}$, by Serre duality we obtain that $h^2 ( \calO_{\bbP^1 \times \bbP^1 \times \bbP^1} ( -2,-2,-2) )$ coincides with $h^{1} (  \calO_{\bbP^1 \times \bbP^1 \times \bbP^1}) = 0$. We conclude $h^1(  \calO_{\tilde{S}} ) =0$. This shows that $\tilde{\calS}$ is a K3 surface.

It remains to show that $\phi : \tilde{\calS} \to \calS$ is a birational equivalence. This follows again by a direct calculation and by linearity of the addition map. Indeed, the set of critical points of the addition map $\Sigma : \partial_a D_1 \times \partial_a D_2 \times \partial_a D_3 \to \bbC^3$ is clearly birational to $\tilde{\calS}$. Moreover, Theorem \ref{thm:degS} implies that this set, regarded as a subvariety of $\bbC^9 = \bbC^3 \times \bbC^3 \times \bbC^3$, is a surface of degree $24$. Since $\Sigma : \bbC^9 \to \bbC^3$ is linear, the degree of the image of (the birational copy of) $\tilde{\calS}$ is at most $24$; moreover, if equality holds, then $\Sigma$ is generically one-to-one \cite[Theorem 5.11]{Mum:ComplProjVars} and it defines a birational equivalence between the critical locus and its image. Since $\deg(\calS) = 24$, we conclude. 
\end{proof}

The subdivision of $\bbR^3$ into its eight orthants induces a subdivision of the boundary of the dice, hence of the set of its exposed points, i.e., $\calS \cap \partial\calD$. Each of these eight regions can be parametrized by the corresponding arcs on two of the three $\partial D_i$'s.

Let $p\in\calS\cap \partial \calD$ be written as $p = \xi_1 + \xi_2 + \xi_3$, with $\xi_i \in \partial D_i$. 
We parametrize the boundaries of the discs via angles $\theta_1,\theta_2,\theta_3$ as follows: 
\begin{align}
\partial D_1 &= \{ (0,0,1)\cos \theta_1 + (0,1,0) \sin \theta_1 : \theta_1 \in [0,2\pi] \}, \\
\partial D_2 &= \{ (1,0,0)\cos \theta_2 + (0,0,1) \sin \theta_2 : \theta_1 \in [0,2\pi] \}, \\
\partial D_3 &= \{ (0,1,0)\cos \theta_3 + (1,0,0) \sin \theta_3 : \theta_1 \in [0,2\pi] \}.
\end{align}
Then the coordinates of $\xi_3$ can be expressed as algebraic functions of the coordinates of $\xi_1$ and $\xi_2$.
More precisely, from the equation of the determinant \eqref{eqn: critical locus phi affine}, we deduce 
\begin{equation}\label{eq_theta3}
\begin{aligned}
\cos \theta_3 &= \pm \frac{|\sin\theta_1 \sin\theta_2 |}{\sqrt{\cos^2\theta_1 \cos^2\theta_2 + \sin^2\theta_1 \sin^2\theta_2}}, \\
\sin \theta_3 &= \pm \frac{|\cos\theta_1 \cos\theta_2 |}{\sqrt{\cos^2\theta_1 \cos^2\theta_2 + \sin^2\theta_1 \sin^2\theta_2}} .
\end{aligned}
\end{equation}

If $(\theta_1, \theta_2) = (k \frac{\pi}{2}, (k+1)\frac{\pi}{2} ) \times (l \frac{\pi}{2}, (l+1)\frac{\pi}{2} )$, then there are exactly two possible choices of the signs in \eqref{eq_theta3} such that $\xi_1+\xi_2+\xi_3 \in \calS$. This subdivides the real points of $\calS$ into $32 = 4 \cdot 4 \cdot 2$ regions. Exactly eight of these regions cover $\calS\cap \partial \calD$ and they are identified by the condition that $\xi_1, \xi_2, \xi_3$ belong to the same (closed) orthant.
\begin{figure}[h!]
\centering
\includegraphics[width=0.45\textwidth]{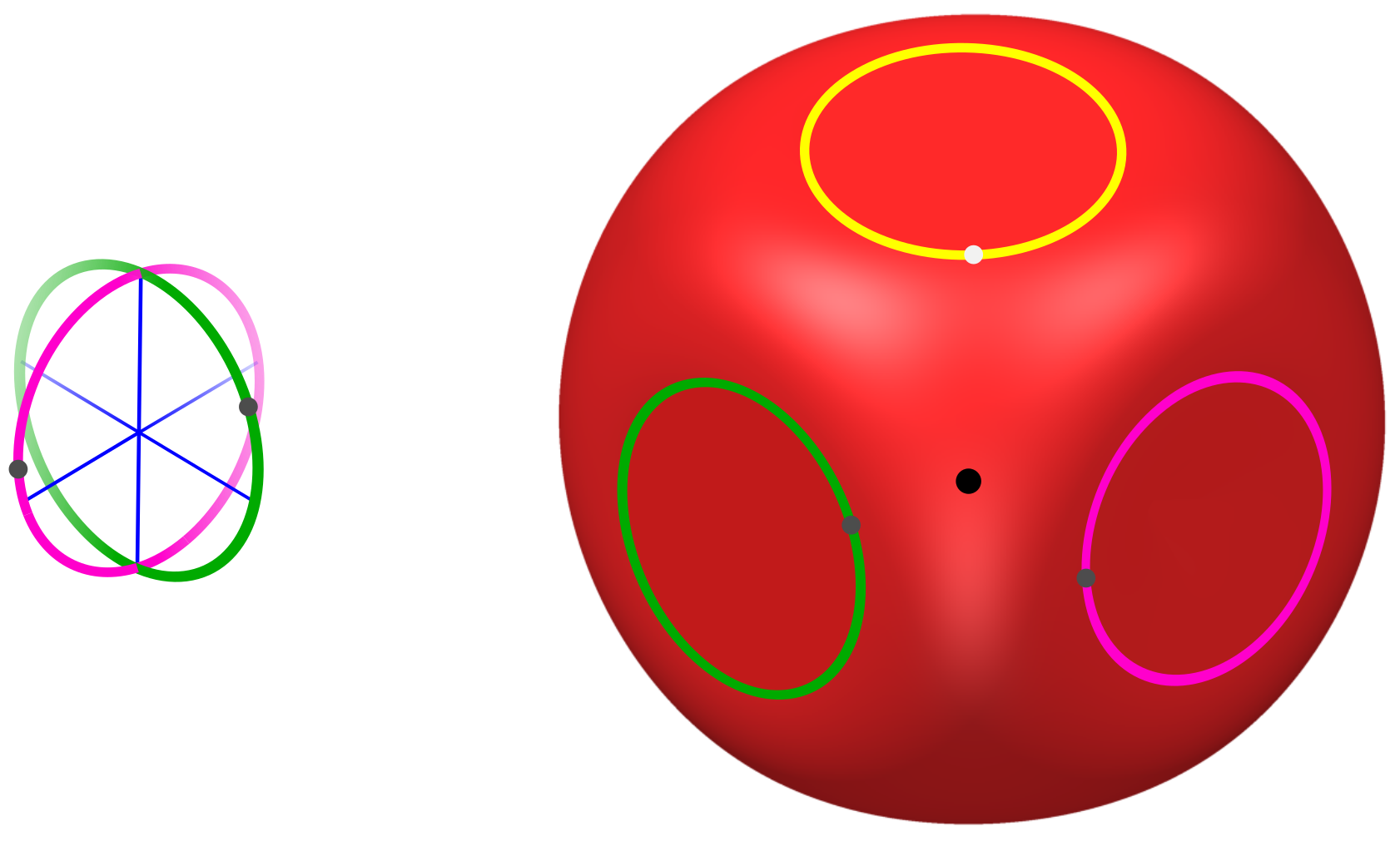}
\caption{Parametrization of $\calS\cap \partial \calD$ with two of the three circles. Given a pair of generic (grey) points on the pink and green circle, there is a unique (white) point on the yellow circle, such that their sum is an extreme exposed point of $\calD$ (black).}
\label{fig:param_dice}
\end{figure}

\section{Conclusions}\label{sec: conclusions}

We summarize our main results concerning the purely nonlinear part of a generic discotope.
\begin{itemize}
 \item $\calS$ is the Zariski closure of the set of exposed points of $\calD$ for the following types:
 \begin{itemize}
 \item[\scalebox{.8}{$\blacktriangleright$}] $\bfN = (0,N, 0 \vvirg 0)$ with $N \geq d-1$;
 \item[\scalebox{.8}{$\blacktriangleright$}] $\bfN = (0 \vvirg 0,N)$.
 \end{itemize}
 \item $\calS$ is irreducible in the following cases:
 \begin{itemize}
 \item[\scalebox{.8}{$\blacktriangleright$}] if \eqref{eq:condition_S} holds with the reverse inequality, in which case $\deg(\calS) = 2^N$;
 \item[\scalebox{.8}{$\blacktriangleright$}] $\bfN = (0,N, 0 \vvirg 0)$ with $N \geq d-1$, in which case $\deg(\calS) \leq 2^N \cdot \binom{N}{d-1}$;
 \item[\scalebox{.8}{$\blacktriangleright$}] $\bfN = (0 \vvirg 0,N)$.
 \end{itemize}
\end{itemize}

In this section, we discuss some open problems, and observations directed toward future work. 

A first question one should address regards an analogue of Theorem \ref{thm:degS} when discs of dimension higher than two are involved. We present an example to explain some of the difficulties.
\begin{example}
Let $D_1 = \{ x_4 = 0, x_1^2 + x_2^2 + x_3^2 = 1 \}$, $D_2 = \{ x_1 = 0, x_2^2 + x_3^2 + x_4^2 = 1 \}$ be two $3$-discs in $\bbR^4$ and let $\calD = D_1+D_2$. This discotope is full dimensional and $\dim \calS = 3$. The ideal of the critical locus of the addition map can be computed via a determinantal method similar to the one discussed in Section \ref{sec:2discs}. We obtain the equation of $\calS$,
\begin{equation}
	x_1^4\!+\!2 x_1^2 x_2^2\!+\!x_2^4\!+\!2 x_1^2 x_3^2\!+\!2 x_2^2 x_3^2\!+\!x_3^4\!-\!2 x_1^2 x_4^2\!+\!2 x_2^2 x_4^2\!+\!2 x_3^2 x_4^2\!+\!x_4^4\!-\!4 x_2^2\!-\!4 x_3^2\! = \!0,
\end{equation}
which is irreducible of degree $4$.
The boundary $\partial \calD$ contains translates of the $3$-dimensional discs: two translated copies of $D_1$ at $x_4 = \pm 1$ and two translated copies of $D_2$ at $x_1 = \pm 1$. The four points of their pairwise intersections are the only points exposed by more than one vector: these points are $(1,0,0,1), (-1,0,0,1), (-1,0,0,-1), (1,0,0,-1)$ and they are, respectively, exposed by the cones
\begin{align}
C_1 &= \{ x_2 = x_3 = 0, x_1 > 0, x_4 > 0 \}, \\
C_2 &= \{ x_2 = x_3 = 0, x_1 < 0, x_4 > 0 \},\\
C_3 &= \{ x_2 = x_3 = 0, x_1 < 0, x_4 < 0 \}, \\
C_4 &= \{ x_2 = x_3 = 0, x_1 > 0, x_4 < 0 \}.
\end{align} 
Notice that for every $i$ and for every $u \in  C_i$, the hyperplane $u^\perp$ contains  $\langle D_1 \rangle \cap \langle D_2 \rangle$, as observed in Example \ref{ex:two_planes_vertices}.
\end{example}

We point out that the determinantal method mentioned above to obtain the equation of $\calS$ is not as straightforward as in the case of $2$-dimensional discs. Implicitly, this method relies on a parametrization of the tangent bundle of the product $\partial_a D_1 \ttimes \partial_a D_N$, in order to impose that the differential of the addition map has submaximal rank. For higher dimensional spheres this parametrization cannot be global since their tangent bundles are not trivial, unlike the case of the circle. 
Nevertheless, in the cases where it can be computed explicitly, the hypersurface $\calS$ is irreducible, hence it is the Zariski closure of the set of exposed points of $\calD$. We propose the following:
\begin{conj}\label{conj:irr_S}
Let $\calD$ be a generic discotope of type $(0,N_2\vvirg N_d)$. Then $\calS$ is irreducible.
\end{conj}
Theorem \ref{thm: joins} proves the conjecture if \eqref{eq:condition_S} holds with the reverse inequality. Remark \ref{rmk:irr_top_discs} proves the statement in the case $(0\vvirg 0,N)$, and Theorem \ref{thm:degS} in the case $(0,N,0 \vvirg 0)$.

In general, we expect the critical locus of $\Sigma$ to be already irreducible, and the addition map $\Sigma$ to be a birational equivalence between $\crit\Sigma$ and $\calS$. Were this true, in the case of $2$-dimensional discs, the upper bound in Theorem \ref{thm:degS} would be an equality. For higher dimensional discs, even under the assumption that the critical locus is irreducible, computing its degree is not trivial and it would be interesting to address it via the classical Giambelli-Thom-Porteous construction, applied to the product of the tangent bundles of the spheres $\partial_a D_i$.

The geometric features highlighted in this work can be used as necessary conditions for a convex body to be a discotope: for instance, there are restrictions for the degrees of the Zariski closure of the set of exposed points. An important future step would be to understand a characterization of discotopes among zonoids or more generally among convex bodies, in the spirit of the zonoid problem. We identify two problems in this direction. 

\begin{problem}\label{problem: characterize discotopes}
 Let $K \subseteq \bbR^d$ be a convex body and $D\subseteq \bbR^d$ be an $n$-dimensional (generalized) disc. Determine whether $D$ is a Minkowski summand of $K$, in the sense that there exists a convex body $K' \subseteq \bbR^d$ such that $K = K' + D$. 
\end{problem}
Problem \ref{problem: characterize discotopes} is understood in the case where $D$ is a disc of dimension $1$, i.e. a segment \cite[Lemma 3.4]{Bol:ClassConvexBodies}. We state the next problem in the language of \cite{BreBueLerMat:ZonoidAlgebra}.
\begin{problem}
Characterize the set of random vectors of $\bbR^d$ whose associated Vitale zonoid is a full dimensional discotope in $\bbR^d$.
\end{problem}
Finally, we expect discotopes not to be spectrahedra, except possibly for small special cases, for instance when $N=1$. However, they are \emph{spectrahedral shadows} \cite{Scheid:SpectraShadows} since they are defined as Minkowski sums of spectrahedra. 
In Section \ref{sec: faces} we observed that $\calD$ is the convex hull of the semialgebraic set $\calS \cap \partial \calD$. 
We propose the following conjecture, which is verified in the cases that we can compute explicitly.
\begin{conj}
The discotope $\calD$ is the convex hull of the real points of $\calS$.
\end{conj}
This would provide examples of real algebraic varieties whose convex hull is a spectrahedral shadow. This topic has been studied for instance in \cite{Scheid:ConvHull, RosStu:dualities} and is related to the Helton--Nie conjecture \cite{HelNie:SemidefiniteRepresentationConvexSets}.
Such questions draw connections between discotopes and the world of convex algebraic geometry, optimization and semidefinite programming.

\bibliographystyle{plain}
\bibliography{biblio}
    
\end{document}